\documentclass[a4paper,11pt]{amsart}
\usepackage[utf8]{inputenc}
\usepackage{amsmath, amssymb, graphics}
\usepackage{amsthm}
\usepackage{amsfonts}
\usepackage{bbm}
\usepackage{abstract}
\usepackage{mathtools}
\usepackage{thmtools}
\usepackage[shortlabels]{enumitem}
\usepackage{hyperref}
\hypersetup{
    colorlinks=true, 
    linkcolor=blue, 
    urlcolor=blue, 
    citecolor=[rgb]{0,0.7,0},
    linktoc=all 
}

\usepackage{tikz}
\usepackage{tikz-cd}
\usepackage{blkarray}
\usepackage{fancyhdr}
\usepackage{nicematrix}
\usepackage[numbers]{natbib}
 \NiceMatrixOptions
  {
    code-for-last-row = \scriptstyle , 
    code-for-first-col = \scriptstyle , 
  }

\newcommand{\be}{\begin{eqnarray}}
\newcommand{\ee}{\end{eqnarray}}

\let\l=\labmda

\let\a=\alpha

\let\l=\lambda

\theoremstyle{definition}
\newtheorem{thm}{Theorem}[section] 

\newtheorem{conj}[thm]{Conjecture}

\newtheorem{prop}[thm]{Proposition}
\newtheorem{lem}[thm]{Lemma}

\newtheorem{cor}[thm]{Corollary}

\newtheorem{dfn}[thm]{Definition}
\newtheorem{ex}[thm]{Example}
\newtheorem{question}[thm]{Question}

\newtheorem{rem}[thm]{Remark}

\newtheorem{convention}[thm]{Convention}
\theoremstyle{remark}

\DeclareMathOperator{\conv}{conv}

\DeclareMathOperator{\lcm}{lcm}

\DeclareMathOperator{\height}{ht}
\DeclareMathOperator{\image}{Im}
\DeclareMathOperator{\volume}{vol}
\DeclareMathOperator{\zz}{\mathbb{Z}}
\DeclareMathOperator{\gcdvar}{gcd}

\DeclareMathOperator{\pyr}{Pyr}

\title{Thin simplices via modular arithmetic}
\author{Vadym Kurylenko}
\email{vadym.kurylenko@ovgu.de}
\begin{document}
\maketitle
\begin{abstract}    The local $h^*$-polynomial is a natural invariant of a lattice polytope appearing in Ehrhart theory and Hodge theory. 
 In this work, we study the question posed in \cite{GKZ}  concerning the classification of lattice simplices with vanishing local $h^*$-polynomial. Such simplices are called thin. We relate this question to linear codes and hyperplane arrangements over finite rings. This allows us to obtain a complete classification of the $4$-dimensional thin simplices, extending the previously known results in dimensions up to $3$. 
\end{abstract}

\section{Introduction}


\subsection{Local $h^*$-polynomial}
Consider a lattice $M \subseteq \mathbb{R}^d$  and a $d$-dimensional lattice polytope $\Delta$.  It is known \cite{ehrhart_sur_1962} that the number of lattice points in an integral dilation  $k \Delta$ is a polynomial in $k$. It is called the Ehrhart polynomial. This information can also be encoded in an invariant known as the $h^*$-polynomial. Namely, we have  \[1+\sum_{k \geq 1} \mid k  \Delta \cap M \mid t^k = \frac{h^*(\Delta,t)}{(1-t)^{d+1}}. \] 

The main object of study in this paper is the \textit{local $h^*$-polynomial}. This is a less-known variant of the $h^*$-polynomial that  takes into account the point counting of the faces of $\Delta$ and the combinatorics of the face lattice of the polytope. In the case of $\Delta$ being a simplex, it can be defined as an alternating sum of $h^*$-polynomials of  the faces of $\Delta$ 
\be  l^*(\Delta,t) \coloneqq \sum_{ \varnothing \subseteq F \subseteq \Delta} (-1)^{\dim \Delta - \dim F} h^*(F,t). \label{lstardef} \ee 
For a simplex it can also be computed as follows. 
Let $\Pi^\circ_\Delta$ be the open parallelepiped spanned by the vertices of $\Delta \times \{1\} \subseteq \mathbb{R}^{d+1} $, then $l^*(\Delta,t)$ is the generating function of the number of points in $\Pi_\Delta^\circ$ with a given last coordinate 
  \[ l^*(\Delta,t) = \sum_{\lambda \in \Pi_\Delta^\circ \cap M \oplus \mathbb{Z}} t^{\lambda_{d+1}}.  \]
 For this reason the local $h^*$-polynomial of a simplex is also known as \textit{box polynomial}. 
  
This notion first appeared in the work of Betke and McMullen \cite{betke_lattice_1985}. It was generalized to general polytopes by Stanley in \cite{stanley_subdivisions_1992}.  Later, it resurfaced independently in the work of Borisov and Mavlyutov \cite{borisov_string_2003}. They, along with Karu \cite{karu_ehrhart_2008}, showed that the coefficients of the local $h^*$ polynomial are non-negative. In  \cite{batyrev_mirror_1996} it was shown  that these coefficients are the top weight Hodge-Deligne numbers of a nondegenerate affine hypersurface in a torus defined by the Newton polytope $\Delta$. 

More recently, the local $h^*$-polynomial has been studied more extensively. In the work \cite{katz_local_2016} Katz and Stapledon described, among other things,  how $l^*(\Delta,t)$ behaves under polyhedral subdivisions.   The local $h^*$-polynomials of the simplices corresponding to the weighted projective spaces were studied in \cite{solus_local_2019} and those of the $s$-lecture hall simplices in \cite{gustafsson_derangements_2020}.
In \cite{villegas_mixed_2019} and \cite{roberts_hypergeometric_2022} $l^*(\Delta,t)$ 
 of circuits appeared in the study of integrality of factorial ratios and hypergeometric motives,  respectively. The work  \cite{bajo_local_2023} discusses the local $h^*$-polynomial of simplices whose coordinates are in Hermitian normal form with one non-trivial row.   

\subsection{Thin simplices}
One of the easy-sounding questions one may ask is when does $l^*(\Delta,t)$ vanish.
This problem 
first appeared in the work of Gel'fand, Kapranov and Zelevinskii
, see \citep[11.4.B]{GKZ}. They coined the following definition. 
\begin{dfn}
	We call a simplex $\Delta$ \textit{thin} if $l^*(\Delta,t) =0$. 
\end{dfn}
Recently this question has been revived and extended to general polytopes by Borger, Kretschmer and Nill in \cite{borger_thin_2023}. They were able to classify the three-dimensional thin lattice polytopes   as well as to characterize the thin Gorenstein polytopes.  Moreover, we refer to their work for a comprehensive review of local Ehrhart theory.



There is a construction that easily produces  thin simplices.  
 Let $\Delta_1$ and $\Delta_2$ be lattice polytopes of dimensions $d_1$ and $d_2$. 
The \textit{free join} of $\Delta_1$ and $\Delta_2$ is 
 \[ \Delta_1 \circ_{\zz} \Delta_2 \simeq \conv \left( \Delta_1 \times \{0^{d_2} \} \times \{ 0\}, \{0^{d_1} \} \times \Delta_2 \times \{1\}\right) \subseteq \mathbb{R}^{d_1+d_2+1}. \]  This construction is particularly nice from the Ehrhart-theoretic point of view since the
  (local) $h^*$-polynomial behaves multiplicatively under free joins, see \cite{henk_lower_2009} for 
 \[ h^*( \Delta_1 \circ_{\zz} \Delta_2,t) = h^*(\Delta_1,t) \cdot h^*(\Delta_2,t)\] 
 and  \cite{nill_gorenstein_2013} for
   \[ l^*( \Delta_1 \circ_{\zz} \Delta_2,t) = l^*(\Delta_1,t) \cdot l^*(\Delta_2,t).\] 
A notable example of a free join is a free join of a polytope $\Delta$ with a single point. 
In this case the resulting polytope $\pyr(\Delta)$ is called a \textit{lattice pyramid} and we have $l^*(\pyr(\Delta),t)=0$ and $h^*(\pyr(\Delta),t)=h^*(\Delta,t)$.  

 Clearly, if a free join is thin, then at least one of its factors must be thin. This leads to a natural question. 
\begin{question}
    Can we classify the thin simplices that are not free joins? 
\end{question}

In dimension two the answer to this question was given in \cite{GKZ}. In this case the only such thin simplex is twice the standard simplex $ 2 \Delta_2$.  In dimension three the answer was provided in \cite{borger_thin_2023}. They deduced that all three-dimensional thin simplices are lattice pyramids. 

In order to answer this question in dimension four we make use of the point of view on lattice simplices presented in \cite{batyrev_lattice_2013} which we review in Section \ref{section:simplices}.
The crucial idea is that for each lattice simplex $\Delta$ we can construct an extended linear code $C_\Delta$ over a finite cyclic ring  $\zz_{N_\Delta}$ of integers modulo $N_\Delta$ for some $N_\Delta \geq1$. This correspondence is in fact one-to-one up to the corresponding isomorphisms.  A linear code $C_\Delta$ of rank $m$  and length $d+1$ can be generated by the rows of an $m \times(d+1) $ matrix $g_\Delta$. Thus, we can reduce the study of a simplex $\Delta$ to the study of a matrix $g_\Delta$ with entries from $\zz_{N_\Delta}$. 

The property of $\Delta$ being thin translates to a simple property of $C_\Delta$.
 The simplex $\Delta$ is thin if and only if the corresponding linear code $C_\Delta$ has no words of maximal weight, i.e. each element of $C_\Delta$ contains a zero.  

One can go further and also use the language of hyperplane arrangements. The columns of the matrix $g_\Delta$ define a hyperplane arrangement $\mathcal{H}_\Delta$ inside $\zz^m_{N_\Delta}$.  
 The thin simplices then correspond to the hyperplane arrangements $\mathcal{H}_\Delta$ that have empty  complement.

\subsection{Main result}
The above point of view allows us to obtain a classification of the four-dimensional thin lattice simplices that are not lattice pyramids or free joins. It is summarized in the following theorem proved in Section \ref{sect: 4d}.  But before proceeding with the theorem let us briefly introduce a few notions from the theory of lattice polytopes.  A lattice polytope is called \textit{spanning} if  its integral points affinely span the ambient lattice.   The \textit{lattice width} of a full-dimensional lattice polytope $\Delta \subseteq \mathbb{R}^d$  is the minimal distance between two parallel hyperplanes in $\mathbb{R}^d$ such that $\Delta$ lies between them.  
\begin{thm} \label{main theorem} 
Let $\Delta$ be a four-dimensional thin simplex that is not a free join. Then $\Delta$ is either one of the $6$ sporadic cases in the Table \ref{sporadictable}  or it belongs to the  one-parameter family of non-spanning simplices of width $1$ given below the table. 
\end{thm}
\begin{table}[h] 
\centering
\renewcommand{\arraystretch}{1.05}
\begin{tabular}{|c|p{4cm}|c|c|c|c|}
\hline
\centering Case & \centering $g_\Delta$ & \centering  $N_\Delta$ &  $h^*(\Delta,t)$ & width & spanning \\ 
\hline
1 & \centering  $\begin{pmatrix}
1 & 1 & 0 & 0 & 0 \\
0 & 1 & 1 & 0 & 0 \\
0 & 0 & 1 & 1 & 0 \\
0 & 0 & 0 & 1 & 1
\end{pmatrix}$  \vspace{5pt}& \centering  2 & $5t^2 + 10t + 1$ &2 & yes \\
\hline
2 &  \centering $\begin{pmatrix}
0 & 0 & 1 & 1 & 1 \\
1 & 2 & 0 & 1 & 2
\end{pmatrix}$ \vspace{5pt} & 3 & $7t^2 + t + 1$ & 1 & no  \\
\hline
3 &  \centering $\begin{pmatrix}
0 & 0 & 1 & 1 & 2 \\
2 & 2 & 1 & 3 & 0
\end{pmatrix}$ \vspace{5pt} & 4 & $5t^2 + 2t + 1$ & 1 & no \\
\hline
4 &  \centering $\begin{pmatrix}
0 & 1 & 1 & 1 & 1 \\
2 & 0 & 1 & 2 & 3
\end{pmatrix}$ \vspace{5pt} & 4 & $t^3 + 11t^2 + 3t + 1$ & 2 & no \\
\hline
5 & \centering $\begin{pmatrix}
0 & 0 & 1 & 1 & 2 \\
2 & 2 & 0 & 2 & 2 \\
0 & 2 & 3 & 3 & 0
\end{pmatrix}$  \vspace{5pt} & 4 & $9t^2 + 6t + 1$ & 2 & yes \\
\hline
6 & \centering $\begin{pmatrix}
4 & 0 & 1 & 2 & 1 \\
4 & 4 & 0 & 4 & 4
\end{pmatrix}$  \vspace{5pt} & 8 & $t^3 + 11t^2 + 3t + 1$ & 2 & no \\
\hline
\end{tabular}
\caption{Sporadic thin simplices for $d=4$.}
\label{sporadictable}
\end{table}
For each even $N_\Delta \geq 2$ there is a thin simplex given by
\be \label{family1} g_\Delta = \begin{pmatrix}
	N_\Delta/2 & 0 & N_\Delta/2 & 0 & 0 \\
	0 & N_\Delta/2 & N_\Delta/2 & 1 & N_\Delta-1
\end{pmatrix} \ee  
with \[h^*(\Delta,t) = \left(\frac{3 N_\Delta}{2}  -1\right) t^2 + \frac{N_\Delta}{2} t +1.\] 

\begin{table}[h]
    \centering
    \begin{tabular}{|c|c|}
        \hline
        Case & Vertices \\
        \hline
        1 & $\left(0,\,0,\,0,\,0\right), \left(2,\,0,\,0,\,0\right), \left(0,\,2,\,0,\,0\right), \left(0,\,0,\,2,\,0\right), \left(0,\,0,\,0,\,2\right)$ \\ 
        \hline
        2 & $\left(-1,\,-1,\,0,\,1\right), \left(0,\,0,\,0,\,0\right), \left(0,\,1,\,1,\,1\right), \left(-1,\,0,\,1,\,-1\right), \left(1,\,-1,\,1,\,0\right)$ \\
        \hline
        3&  $\left(-1,\,1,\,0,\,0\right), \left(0,\,0,\,0,\,0\right), \left(-1,\,-1,\,-1,\,1\right), \left(1,\,1,\,-1,\,1\right), \left(0,\,0,\,1,\,1\right)$ \\
        \hline 
        4 & $\left(-1,\,1,\,-1,\,-1\right), \left(-1,\,-1,\,0,\,-1\right), \left(0,\,0,\,1,\,1\right), \left(1,\,1,\,0,\,-1\right), \left(0,\,0,\,-1,\,1\right)$ \\ 
        \hline
        5 & $\left(-1,\,0,\,1,\,-1\right), \left(0,\,2,\,-1,\,-1\right), \left(-1,\,0,\,-1,\,1\right), \left(1,\,0,\,-1,\,1\right), \left(0,\,0,\,1,\,1\right) $  \\ 
        \hline
        6 & $\left(0,\,0,\,0,\,-1\right), \left(0,\,-1,\,-1,\,-1\right), \left(1,\,-1,\,1,\,1\right), \left(-1,\,0,\,0,\,1\right), \left(1,\,1,\,-1,\,1\right)$ \\ 
        \hline
        family & $\left(-1, 0, 0, 0\right), \left(1, -1, -3, 2\right), \left(1, 0, 0, 0\right), \left(0, N/2, -N/2+1, 0\right), \left(0, 0, 1, 0\right)$ \\
        \hline
    \end{tabular}
    \caption{Vertices of the thin simplices from Theorem \ref{main theorem}.}
    \label{table:vertices}
\end{table}

We see that, contrary to the dimensions $\leq3$, in dimension $4$ there are infinitely many thin simplices that are not free joins.   


\subsection{Experiments and questions}
By using linear codes instead of simplices, it becomes somewhat easier to produce examples of thin simplices. By considering generating matrices of extended linear codes  over $\zz_{N_\Delta}$ for $N_\Delta$ small and with only a few rows, one can relatively quickly obtain a small database of thin simplices. 
The corresponding SageMath \cite{the_sage_developers_sagemath_2022} code is available online\footnote{\href{https://doi.org/10.5281/zenodo.18016483}{https://doi.org/10.5281/zenodo.18016483}}. 
Using the criterion from Lemma  \ref{FJ_criterion} we can quickly establish which simplices are definitely not free joins. In Table \ref{table: higher} the found data is presented, and further explanations about the search parameters are provided in Appendix \ref{appendix_data}. 
This is by far not a  classification. 
\begin{table}[h]
	\centering
	\begin{tabular}{|c|c|c|c|c|c|c|}
		\hline $d$ & $\#$& spanning & width = 1& width = 2 & empty &  $ \deg h^*( \Delta,t) \leq d/2$. \\
		\hline
		$5$ & 69  & 0 & 69  & 0 & 5 & 67   \\ 
		\hline
		$6$  & 704  & 4  & 655  & 49 & 35 & 541   \\ 
		\hline
		$7$ & 1071  & 0  & 1071 & 0 & 130&  1053   \\ 
		\hline
	\end{tabular}
	\caption{Thin simplices that are not free joins found experimentally}
	\label{table: higher}
\end{table}

It is interesting to look at different subclasses of thin simplices, as suggested in \cite{borger_thin_2023}.  An important class of lattice polytopes consists of spanning lattice polytopes.
Note  the scarcity of spanning thin simplices in the above table. 
In dimension $4$ we also have only $2$ spanning thin simplices.  Furthermore, it is somewhat surprising that we were not able to find any spanning thin simplices in dimensions 5 and $7$. This naturally leads to the following questions. 
\begin{question}
Are there finitely many spanning thin simplices that are not free joins in each dimension? Are there any in the odd-dimensional case? 
\end{question}
\noindent We discuss spanning thin simplices more in Subsection \ref{subsec:spanning}. 

Another important  property of a lattice polytope $\Delta$  
is its lattice width.  As discussed in \cite{borger_thin_2023}, it seems that thin lattice polytopes tend to have small width. 
To compute the lattice width we first make use of Proposition \ref{prop: Cayley}  from Subsection \ref{subsec:width} to single out width $1$ simplices and then we use Polymake \cite{assarf_computing_2017} to compute the width of the remaining examples. One can see that the number of thin simplices of width $\geq 2$ is rather small and in odd dimensions we were not able to find any. 
We would like to state a few conjectures in the strongest form possible. 
\begin{conj}\label{conj:width1}
    In each even dimension there are only finitely many thin simplices of lattice width $\geq 2$. 
\end{conj}
\begin{conj}\label{conj:width2}
	In odd dimensions all thin simplices have width $1$.
\end{conj}
\noindent  One can  also ask whether these conjectures hold for thin polytopes, that are not simplices.   As was shown in \cite{borger_thin_2023}, all three-dimensional thin polytopes have width $1$.


In \cite{borger_thin_2023} a question was asked whether there are thin \textit{empty} simplices in dimensions $\geq 5$ that are not lattice pyramids.  Recall that a \textit{quotient group} of a lattice simplex $\Delta$ is defined as the quotient of $\mathbb{Z}^{d+1}$ by the  subgroup generated by the vertices of $\Delta \times \{1\}$.   Proposition 3.18 in \cite{borger_thin_2023} implies that a thin empty simplex that is not a lattice pyramid must have a non-cyclic quotient group. We answer the question above  positively.   All the empty simplices considered in Table \ref{table: higher} are not lattice pyramids and  have a non-cyclic quotient group. Moreover, all of them have width $1$.  Here is an example. 
\begin{ex}
 The $6$-dimensional simplex, whose vertices are the columns of
\[ \left(\begin{array}{rrrrrrr}
	0 & 1 & 0 & 0 & 1 & 3 & 1 \\
	0 & 0 & 1 & 0 & -1 & -1 & 2 \\
	0 & 0 & 0 & 1 & -3 & -6 & -1 \\
	0 & 0 & 0 & 0 & 2 & -6 & -2 \\
	0 & 0 & 0 & 0 & 0 & 10 & -2 \\
	0 & 0 & 0 & 0 & 0 & -2 & 2
\end{array}\right), \]
is thin and empty with $h^*(\Delta,t)= 3 t^4 + 12 t^3 +16 t^2 +1$. Its quotient group is $ \zz_2 \times \zz_2 \times \zz_8$.  
The corresponding linear code over $\zz_8$ can be generated by 
\[ \left(\begin{array}{rrrrrrr}
	0 & 3 & 3 & 4 & 4 & 5 & 5 \\
	4 & 0 & 0 & 0 & 4 & 4 & 4 \\
	0 & 0 & 4 & 4 & 4 & 0 & 4
\end{array}\right).\]
\end{ex}

Finally, a thin polytope $\Delta$ is called \textit{trivially thin}  if $ \deg h^*( \Delta,t) \leq d/2$. 
We see that most of the found thin simplices that are not free joins are non-trivially thin.
\subsection{Structure of the paper} In Section \ref{section:simplices} we review the basics of Ehrhart theory of simplices and explain the relationship between simplices, linear codes, and hyperplane arrangements. Using this connection, we derive certain necessary conditions for a simplex to be thin.  We continue the section by focusing on the case of spanning thin simplices. In the end of the section we discuss how to characterize lattice simplices of width 1 in terms of the corresponding linear code.  Based on the above results, in Section \ref{sect: 4d} we classify the four-dimensional thin simplices.

\subsection*{Acknowledgements} I would like to thank my supervisor Fernando Rodriguez Villegas for many useful discussions and, in particular,  for his suggestion to look at linear codes in a related problem.  I am indebted to Benjamin Nill for his interest in this work and a lot of  valuable comments. I extend my thanks to Asem Abdelraouf, Mykyta Bulakhov and Giulia Gugiatti for useful conversations. I am grateful to the anonymous reviewers for their helpful suggestions, which improved the presentation of this work. Moreover, I would like to thank the Armed Forces of Ukraine for keeping my family safe in Kharkiv.

\section{Lattice simplices and modular arithmetic}
\label{section:simplices}

\subsection{Our Setup}
For a general lattice polytope $\Delta$ the coefficients of the $h^*$-polynomial do not have a known simple combinatorial interpretation. However, if $\Delta$ is a simplex, then one can consider the following. 
 Let us denote the vertices of the simplex by $v_0, \ldots, v_d$ and let $(v_i,1)$ be the lift of $v_i$ to $M \oplus \mathbb{Z}$. Define the half-open parallelepiped associated to $\Delta$ as
\[ \Pi_\Delta =   \left\{  \sum_{i=0}^{d} a_i (v_i,1)  , \quad 0 \leq a_i <1
 \right\}.   \]  
The $h^*$-polynomial of $\Delta$ can then be expressed as a sum over the integral points inside $\Pi_\Delta$: 
 \[ h^*(\Delta,t) = \sum_{\lambda \in \Pi_\Delta \cap M \oplus \mathbb{Z}} t^{\lambda_{d+1}}. \]
In other words, the $k$th coefficient of $h^*(\Delta, t)$ equals the number of integral points in $\Pi_\Delta$ that lie at height $k$.

 Betke and McMullen \cite{betke_lattice_1985} suggested to look at  a slight modification of this. Consider the open parallelepiped 
 \[ \Pi_\Delta^\circ =   \left\{  \sum_{i=0}^{d} a_i (v_i,1)  , \quad 0 < a_i <1
 \right\}   \] 
and the polynomial
     \[ l^*(\Delta,t) = \sum_{\lambda \in \Pi_\Delta^\circ \cap M \oplus \mathbb{Z}} t^{\lambda_{d+1}}.  \]
It is called the local $h^*$-polynomial of $\Delta$. 
 Since it enumerates the number of points inside a parallelepiped, that is, a box, it is also known as \textit{box polynomial}. Let us discuss some of its properties.
 From this definition it is evident that the coefficients of $l^*(\Delta,t)$ are non-negative numbers. In particular, the linear coefficient is the number of lattice points in the interior of the simplex 
\[ l^*_1(\Delta) = \mid \Delta^\circ \cap M \mid. \] 
Therefore, if $\Delta$ is thin, then it must have no interior lattice points, i.e. it is hollow. 

Since on  $\Pi_\Delta^\circ$ there is an involution given by \[ \sum a_i (v_i,1) \rightarrow \sum (1-a_i) (v_i,1), \]  the polynomial $l^*(\Delta,t)$ is palindromic, namely $l^*(\Delta,t) = t^{d+1} l^*(\Delta, t^{-1})$. 

Note that  these properties still hold if $\Delta$ is not a simplex.

\subsection{Linear codes} We are going to relate lattice simplices to linear codes, so here are the necessary definitions and facts from  coding theory. 
Let $\mathbb{Z}_N$ be the ring of integers modulo $N$.
\begin{convention}
In this paper we always identify the elements of $\mathbb{Z}_N$ with their representatives $0,1, \ldots, N-1$.  In particular, this allows us to speak about $\gcd$ of the elements of $\zz_N$, by which we mean the usual integer $\gcd$ of the corresponding representatives. 
\end{convention}

\begin{dfn}
A \textit{linear code}  of length $n$ over $\mathbb{Z}_N$ is a submodule $C$ of the free module $\mathbb{Z}_N^n$.
\end{dfn}
A standard way to produce a linear code is to take an $m\times n$ matrix $g$ over $\mathbb{Z}_N$ and consider the submodule of $\mathbb{Z}_N^n$ generated by the rows of $g$.  The elements of a linear code are called \textit{words}. For a word $c \in C$ its \textit{weight}  is the number of non-zero entries  
\[ w(c) \coloneqq | \{ i \: : \: c_i \neq 0 \} | .\]
The generating function of weights is called \textit{weight enumerator}
\[ W_C(X)=\sum_{c \in C} X^{w(c)}.\]


We call two linear codes $C_1$ and $C_2$ of length $n$ \textit{isomorphic} if after a possible permutation of indices the sets of words coincide, that is, there exists a permutation $\sigma \in S_n$ such that 
 \[ \{ c \: : \: c  \in C_1 \} = \{ (c_{\sigma(1)}, \ldots, c_{\sigma(n)} )\: : \: c  \in C_2 \}. \]

Clearly, isomorphic linear codes have the same weight enumerator. Moreover, if a linear code is defined by a generating matrix, then any permutation of its rows and columns defines an isomorphic linear code.

\begin{dfn}
 We call a linear code $C$ \textit{extended} if for every word $c\in C$ we have $\sum c_i =0 \mod  N$.   
\end{dfn}
In this paper we use only extended linear codes. 
For an extended linear code $C$ we can  define the \textit{height} of a word $c \in C$ to be 
\[ \height(c) \coloneqq \frac{1}{N} \sum_{i=1}^n c_i. \]

\subsection{From simplices to linear codes}

Let $\Delta \subseteq M \simeq \mathbb{Z}^d$ be a $d$-dimensional lattice simplex with vertex set $v_0, \ldots, v_{d}$. Following \cite{batyrev_lattice_2013} one can associate the following additive group to a simplex:
\[ \Lambda_\Delta \coloneqq \left\{ (x_0,\ldots, x_d) \in (\mathbb{R}/\mathbb{Z})^{d+1} \: :  \quad \sum_{i=0}^d \{ x_i \} (v_i,1) \in M \oplus \mathbb{Z}  \right\}, \] 
where $\{ \cdot \}: \mathbb{R} \rightarrow [0,1)$ is the fractional part function, i.e. $ \{ x\} =  x - \lfloor x \rfloor $, where $\lfloor x \rfloor $ is the biggest integer that is smaller than or equal to $x$. 

We can consider the group $\Lambda_\Delta$  as a linear code. 
Let $N_\Delta$ be the least common multiple of the denominators of $x_i$'s of $\Lambda_\Delta$. Since the order of $\Lambda_\Delta$ is the normalized volume of the simplex, $N_\Delta$ is a divisor of $\volume_{\zz}(\Delta)$. Given an element $(x_0, \dots, x_d)$, we take the representative $( \{x_0 \}, \ldots, \{ x_d\})$ and multiply it by $N_\Delta$. In this way we can promote each element of $\Lambda_\Delta$ to a word of an extended code over $\mathbb{Z}_{N_\Delta}$. Let us denote this linear code by $C_\Delta$
\[ C_\Delta  \coloneqq  \left\{ c= (c_0, \ldots, c_d) \in (\zz_{N_\Delta})^{d+1} \: : \: \sum c_i (v_i,1)=0 \mod N_\Delta  \right\}. \]
For a word $c \in C_\Delta$ consider $ \gcd(c,N_\Delta) = \gcd( c_0, \ldots, c_{d}, N_\Delta)$. 
Since $N_\Delta$ was chosen to be minimal, we have  
 \[    \gcd \left(  \left\{  \gcd (c,N_\Delta)    \right\}_{c \in C_\Delta}     \right)  =1,    \]
i.e. the greatest common divisor of $N_\Delta$ and all the entries of all the words is $1$.

Recall that two lattice simplices $\Delta_1$ and $\Delta_2$ are \textit{isomorphic} if there is an affine unimodular transformation of the ambient lattice $M$ mapping $\Delta_1$ to $\Delta_2$.  The following theorem is crucial since it allows us to talk interchangeably about simplices and linear codes. 
\begin{thm}[\cite{batyrev_lattice_2013}, Theorem 2.3.] \label{thm:BH}
Up to the corresponding isomorphisms there is a one-to-one correspondence between $d$-dimensional lattice simplices and extended linear codes $C$ of length $d+1$ over $\zz_N$ for $N \geq 1$ such that the greatest common divisor of all the entries of all the words in $C$ and $N$ is $1$. 
\end{thm}



So far, we only described how from a given simplex one constructs the corresponding linear code. Let us also describe the inverse construction. 
Let $C$ be an extended code over $\mathbb{Z}_N$ of length $d+1$. Consider the natural projection map $\pi: \mathbb{Z}^{d+1} \rightarrow \zz_N^{d+1}$.   The preimage $M\coloneqq\pi^{-1}(C)$ is a sublattice of $\zz^{d+1}$. 
 Let $ e_0, \ldots, e_d$ be the standard basis of $\zz^{d+1}$. Define the simplex
  \[ \Delta_C=\conv( N e_0, \ldots , N e_d) \] with respect to the affine lattice $ \text{aff}(N e_0, \ldots, N e_d) \cap M$.  This gives us a map $C \rightarrow \Delta_C$.

\begin{ex} \label{2d2example}
Consider twice the standard two-dimensional simplex $\Delta= 2 \Delta_2$. As was already mentioned in the introduction, this is the only interesting thin simplex in dimension two. The coordinates of its vertices are $(0,0), \: (2,0), \: (0,2)$. The corresponding group $\Lambda_{\Delta}$ has $\volume_{\zz} (\Delta) = 4$ elements, namely 
 \[ \left\{\left(0,0,0\right), \left( \frac{1}{2}, \frac{1}{2},0 \right), \left( \frac{1}{2}, 0, \frac{1}{2} \right), \left(0, \frac{1}{2}, \frac{1}{2}\right) \right\}. \]
 We see that $N_\Delta=2$ and the corresponding linear code $C_{\Delta}$   is a linear code over $\zz_2$ with 4 words 
  \[ \left\{\left(0,0,0\right), \left( 1, 1,0 \right), \left( 1, 0, 1\right), \left(0, 1, 1\right) \right\}. \]
This linear code  can be generated by the following matrix 
\[ g_{ \Delta} = \begin{pmatrix}
    1 & 1 & 0 \\
    0 & 1 & 1 
\end{pmatrix}. \]
Let us now  describe how to go back to the simplex from this linear code. The lattice $M$ is generated by $f_0=(1,1,0), f_1 = (0,1,1)$ and $f_2=(1,0,1)$.
The affine space $\text{aff}(2 e_0, 2 e_1, 2 e_2)$ is the two-dimensional affine space given by $\{ x \in \mathbb{R}^3 \: : \: x_0+x_1+x_2 =2 \}$. Thus, $ \text{aff}(2 e_0, 2 e_1, 2 e_2) \cap M$ is the affine sublattice of $M$  given by 
\[ \left\{ a f_0 + b f_1 + c f_2 \: : \: (a,b,c) \in \zz^3, a+b+c=1 \right\} \subseteq M.  \] 
The vertices of the simplex are given by $(2,0,0), (0,2,0), (0,0,2)$ in $M$, and thus in $ \text{aff}(2 e_0, 2 e_1, 2 e_2) \cap M$ they are given by $(1,-1,1), (1,1,-1),(-1,1,1)$.  
After rotating everything with the matrix $((1,1,0),(0,1,1),(1,1,1))$ we can project to the first two coordinates and obtain the simplex $ 2 \Delta_2$ in its usual  form.
\end{ex}

Note that 
the height of a word $c\in C_\Delta$  corresponds to the height of the vector $\sum_{i=0}^{d} \frac{c_i}{N_\Delta }  (v_i,1)$ in the half-open parallelepiped of $\Delta$. 
Therefore, we can calculate the $h^*$ -polynomial of a simplex $\Delta$ via the corresponding linear code 
 \[ h^*(\Delta,t) = \sum_{c \in C_\Delta} t^{\height(c)}. \]
 Moreover, the words that have no zeros correspond exactly to the points of $\Pi^\circ_\Delta$, giving us an expression for the local $h^*$-polynomial
 \[   l^*(\Delta,t)= \sum_{\substack{c \in C_\Delta \\ w(c)={d+1}}} t^{\height(c)}.\]
From this we see that the thinness of the simplex $\Delta$ translates into a nice property of the corresponding linear code. The sum defining $l^*(\Delta, t)$ is empty if there are no words of weight $d+1$ in $C_\Delta$. This motivates the following definition.  
\begin{dfn}
    An extended linear code $C$ is called \textit{thin} if it contains no words of maximal weight. 
\end{dfn}
\begin{cor}
\label{prop:thincode}
    A lattice simplex $\Delta$ is thin if and only if the corresponding linear code $C_\Delta$ is thin. 
\end{cor}
\noindent Note, that the linear code considered in the Example \ref{2d2example} above is thin. 

It is easy to see \cite{batyrev_lattice_2013} that the lattice simplex $\Delta$ is a lattice pyramid if and only if the corresponding linear code is \textit{degenerate}. This means that there exists an index $i \in \{0,\ldots,d\}$ such that every word $c \in C_\Delta$ has $c_i=0$. In other words, all the generating matrices of $C_\Delta$ have a column of zeros.

We also have to translate the construction of free joins to the language of linear codes.  Consider first the following definition. 
\begin{dfn}
    Suppose $C_1$ and $C_2$ are linear codes over $\zz_{N_1}$ and $\zz_{N_2}$ correspondingly. A \textit{direct sum} of $C_1$ and $C_2$ is the linear code $C_1 \oplus C_2$  over $\zz_{\lcm(N_1,N_2)}$ whose words are given by 
    \[ \left\{  \frac{\lcm(N_1,N_2)}{N_1} \cdot  c_1 \: \huge{|} \: \frac{\lcm(N_1,N_2)}{N_2} \cdot  c_2 \: : \: c_1 \in C_1, \: c_2 \in C_2   \right\},  \] 
where $ |$ is the usual  concatenation. 
This construction corresponds exactly to the free joins on the simplices side. 
\end{dfn}
\begin{lem}
    A simplex $\Delta$ is a free join of the simplices $\Delta_1$ and $\Delta_2$ if and only if the corresponding linear code $C_\Delta$ is a direct sum of the linear codes $C_{\Delta_1}$ and $C_{\Delta_2}$. 
\end{lem}
\begin{proof}
It is clear from the construction that if a simplex is a free join, then the corresponding linear code is a direct sum. Suppose we have a linear code $C$ over $\zz_N$ which is a direct sum $C = C_1 \oplus C_2$.  Let $C_1=C_{\Delta_1}$ and $C_2=C_{\Delta_2}$.
We know that the simplex $\Delta_1 \circ_{\zz} \Delta_2$ corresponds to $C_{\Delta_1} \oplus C_{\Delta_2}$. By the fact that the correspondence in Theorem \ref{thm:BH} is one-to-one we can conclude that $\Delta_C \simeq \Delta_1 \circ_{\zz} \Delta_2$.

\end{proof}

\begin{lem}
	\label{FJ_criterion}
	
	If a lattice simplex $\Delta$ is a free join, then both the $h^*$-polynomial and the weight enumerator of the corresponding linear code factorize.  However, the opposite is not true. 
\end{lem}

\begin{proof}
	The factorization of the $h^*$-polynomial for free joins was already discussed in the introduction.   Let $C_\Delta = C_1 \oplus C_2$ be the corresponding linear code. Since the words of $C_\Delta$ are concatenations (with factors) of all the possible pairs of words $(c_1,c_2) \in C_1 \times C_2$,  the weight of a word in $C$ is exactly the sum of the weights of the corresponding words $c_1$ and $c_2$.

 The failure of the opposite implication can be shown by the following example.  Consider the code $C$ of length $7$ over $\zz_4$ generated by 
 \[ \left(\begin{array}{rrrrrrr}
 	0 & 1 & 2 & 2 & 2 & 2 & 3 \\
 	2 & 3 & 0 & 0 & 0 & 2 & 1
 \end{array}\right). \] 
 We have 
 \[ W_C(X)= (X^2+1) ( 3 X^4+1), \quad h^*(\Delta_C,t)=(t+1) (3 t+1). \]
 This code is non-degenerate, so if it is a direct sum it must have two non-degenerate factors. In particular, one of the factors must correspond to a simplex that is not a lattice pyramid and has $h^*(\Delta,t)=t+1$. Using the classification of degree 1 lattice polytopes in \cite{batyrev_multiples_2007} we can deduce that this factor must correspond to the $1$-dimensional simplex $[0,2]$. Therefore, the second factor must correspond to a $5$-dimensional simplex with $h^*(\Delta,t)= 3 t+1$ that is not a lattice pyramid. Using the same classification, one can deduce that this is not possible. 
\end{proof}

The above lemma  gives a useful sufficient condition for deciding if a simplex $\Delta$ is a not a free join. Namely, if  at least one of the polynomials $W_{C_\Delta}(X)$ and $h^*(\Delta_C)$ does not factorize into polynomials with non-negative coefficients, then $\Delta$ is not a free join.

\subsection{From simplices to hyperplane arrangements} 
Let $g$ be an $m \times n$ generating matrix of a linear code $C$ over $\mathbb{Z}_N$. Let $g_{i}$ be the $i$th column of $g$. Define 
    \[ H_i = \left\{ x \in \mathbb{Z}_N^m \: \mid \: g_{i} \cdot x =0 \mod N  \right\}.  \]
We call this set the \textit{hyperplane} defined by $g_{i}$. This way from a linear code we get a hyperplane arrangement over $\zz_N$ \[ \mathcal{H}_C= \{ H_0, \ldots, H_{d} \}. \] 
If the linear code is the code $C_\Delta$ coming from a lattice simplex, let us denote this hyperplane arrangement as $\mathcal{H}_\Delta$.  The following proposition is going to be the central tool for classifying the four-dimensional thin simplices. 
\begin{cor}
    A simplex $\Delta$ is thin if and only if the complement of $\mathcal{H}_\Delta$ is empty, i.e. $\zz_N^m = H_0 \cup H_1 \cup \ldots \cup  H_d$. 
\end{cor}
\begin{proof}
Suppose $g_\Delta$ is a generating matrix of $C_\Delta$ and it has $m$ rows. Then the words of $C_\Delta$ can be obtained by multiplying $g_\Delta$ from the left by all the possible $a=(a_1, a_2,\ldots, a_m) \in \zz_N^m$.  Now, clearly, if every $a$ belongs to $\mathcal{H}_\Delta$, then every word has a zero at some place. Furthermore, if a word has a zero at the $i$th place, it means that the corresponding points $a$ belong to the hyperplane $H_i$. Thus, if every word has a zero somewhere, then all the points $a$ belong to some hyperplane. 
\end{proof}

For a hyperplane $H_i$ defined by the column $g_{i}$ of the $m \times (d+1)$ generating matrix $g_\Delta$  define 
\be \label{gcd} \gcdvar_i \coloneqq \gcd(g_{1i}, \ldots, g_{mi}, N). \ee 
Then the hyperplane $H_i$ contains  $ \gcdvar_i \cdot N^{m-1}$ points. This can be deduced as follows. Let $G_i : \zz^m_N \rightarrow \zz_N$ be the $\zz$-module homomorphism,  defined by the column $g_i$, that sends $x$ to $g_i \cdot x$.   The image of this map is exactly the subgroup of $\zz_N$ generated by $\gcdvar_i$, in particular $| \image G_i| = N/\gcdvar_i$. By the first isomorphism theorem we have $ | H_i  | = |  \ker G_i  | =  N^m/ | \image G_i| = \gcdvar_i  \cdot N^{m-1}$.

We give a necessary condition for the code $C$ to be thin. 
\begin{prop} \label{prop:gcdsum}
Let $C$ be a  linear code of length $d+1$ over $\zz_N$ and $g$ be its generating matrix. 
If the linear code $C$ is thin, then
    \[ \sum_{i=0}^d \gcdvar_i >N.  \] 
\end{prop}
\begin{proof}
Suppose $g$ has $m$ rows.  Let$ \{H_0, \ldots, H_d\}$ be the corresponding hyperplane arrangement. 
All the hyperplanes contain at most  $ \sum_{i=0}^{d}   \gcdvar_i   \cdot N^{m-1}$ points. Since the point $(0, \ldots, 0) \in \zz^m_N$ belongs to all of them, we arrive at the strict inequality 
\[ \sum_{i=0}^{d}   \gcdvar_i   \cdot N^{m-1} > N^m. \]
\end{proof}
In principle, one can continue with the inclusion-exclusion procedure to account for the intersections of subsets of the hyperplanes. But the expressions for the number of points of these intersections become more difficult and   it does not seem practical to proceed this way. 
\begin{cor} \label{cor:dimN}
    If $C$ is thin and every $\gcdvar_i=1$ (in particular, if $N$ is prime), then  $d \geq N$.
\end{cor}
\begin{cor} \label{cor : smallest prime}
    Let $p$ be the smallest prime in the factorization of $N$. If $C$ is thin, then $d \geq p$. 
\end{cor}
\begin{proof}
Let $N = p_1^{m_1} \cdot \ldots \cdot p_l^{m_l}$ be the prime factorization of $N$ with $p_1 < \ldots < p_l$. Since each $\gcdvar_i$ is smaller than $N$, there exists a possibly non-unique  $j_i$, such that  $\gcdvar_i \leq N/ p_{j_i}$.  Thus, there exist integers $\a_j \geq 0$ such that
\[ N < \sum_{i=0}^d \gcdvar_i  < \sum_{j=1}^l  \a_j   \frac{N}{p_j}, \]
and $\sum \a_j = d+1$.  Dividing everything by $N$ and using that $p_1$ is the smallest prime in the factorization we obtain 
\[ 1 < \frac{d+1}{p_1}.\]

\end{proof}

\subsection{Spanning thin simplices} \label{subsec:spanning} A lattice polytope $\Delta$ is called \textit{spanning } if every point in $M \oplus \zz$ is an integer linear combination of the lattice points in $\Delta \times \{1\}$.   Spanning lattice simplices have a nice characterization in terms of  linear codes. 
\begin{prop}
    A lattice simplex $\Delta$ is spanning if and only if the corresponding linear code $C_\Delta$ can be generated by a matrix $g$ whose rows have height one.
\end{prop}
\begin{proof}
    The "only if" direction is straightforward. Since $\Delta$ is spanning, every word in $C_\Delta$ is a combination of words that correspond to the lattice points of the simplex, this gives a generating matrix with height one rows. 

    For the "if" direction consider the half-open parallelepiped. Since $C_\Delta$ is generated by the words of height one, it means that any point inside $\Pi_\Delta$ is a linear combination with integral coefficients of the points at height one. Since we can cover $M\oplus\zz$ by translating $\Pi_\Delta$ by multiples of $(v_i,1)$ it follows that any lattice point is expressible through the points of $\Delta \times \{1\}$, i.e. $\Delta$ is spanning. 
\end{proof}
\begin{rem}
    Note that if $\Delta$ is spanning  we might need a lot of rows  to have a generating matrix with all rows having height $1$. There is a universal bound on this number, namely one needs at most  $(d+1) 2^{(d+1)}$ rows, see \cite{flatnessconstant} and \cite{eisenbrand}. 
    
     For example, the simplex corresponding to the linear code generated by 
    \[ \begin{pmatrix}
        3 & 3 & 0 & 0 &0 \\
        3 & 0 & 3 & 2 & 4 
    \end{pmatrix} \] 
    over $\zz_6$ is spanning since this code can also be generated by 
     \[ \begin{pmatrix}
        3 & 3 & 0 & 0 & 0 \\
        3 & 0 & 3 & 0 & 0 \\ 
        0 & 0 &0& 2 & 4 
    \end{pmatrix}.   \] 
However, this linear code cannot be generated by a matrix with two rows with height $1$. 

Nevertheless, when $N$ is prime and a linear code $C$ of dimension $m$ corresponds to a spanning simplex, then this code can be generated by a matrix with exactly $m$ rows with height $1$. 
\end{rem}

As we mentioned in the introduction we expect that spanning thin simplices are relatively rare. It seems that it is particularly so for the simplices with prime $N_\Delta$. In particular, computations in low dimensions ($d\leq 8$) suggest the following conjecture.
\begin{conj}
    For prime $N \geq 3$  there are no non-degenerate thin codes corresponding to spanning simplices. 
\end{conj}

The following proposition shows that this conjecture is true for the linear codes of dimension $2$. Moreover, it gives some weak bounds for the  codes of higher dimension.

\begin{thm} Let $\Delta$ be a spanning thin simplex and $C$ be the corresponding linear code of dimension $m\geq 2$ over $\zz_N$, with $N$ prime.  Then for each $m\geq 2$ there exists a constant $N_m$ such that  $N \leq N_m$. In particular, for small $m$ the values of $N_m$ are in the table below.  \label{thm:spanning}
\begin{table}
	\centering
\begin{tabular}{|c|c|c|c|c|c|c|} 
	\hline
	$m$ & 2 & 3 & 4 & 5 & 6 & 7  \\ \hline
	$N_m$ & 2 & 17 & 83 & 379 & 1499 & 5987  \\ \hline
\end{tabular}
\vspace{0.3cm}
\caption{Values of the constants $N_m$ appearing in Theorem~\ref{thm:spanning}. }
\label{tab:Nm} 
\end{table}

\end{thm}
\begin{proof}
Since $N$ is prime, every column has $\gcdvar_i=1$.  Since the code can be generated by a matrix $g$ with $m$ rows of height $1$, the sum of all the entries of the generating matrix is $ m N$.  Suppose there are $r$ distinct hyperplanes in $\mathbb{Z}^m_{N_\Delta}$, then they contain at most $r \cdot N^{m-1}$ points, and thus $r \geq N+1$ since $\Delta$ is thin.

 We want to estimate the sum of the entries of the normals that define these hyperplanes, that is, the sum of the entries of $g$. We will use this estimate to show that for a fixed $m$ starting from some $N$ the sum of the entries of $g$ with $N+1$ different columns can only be larger than $m N$. 

Let $H$ be a hyperplane in $\zz_N^m$ defined by a normal $h$.  Let us call $h$ minimal and denote it $h_{min}$ if the sum of its entries is the smallest possible among all the normals defining the same hyperplane. For example, consider $N=7$, $m=2$ and the hyperplane defined by the normal $h=(2,3)$. This $h$ is not minimal and the minimal one is given by $h_{min}=(3,1) = 5 \cdot(2,3)$.   Note that a minimal $h$ may not be unique.  For example,  if we consider the hyperplane defined by $(6,1)$, then all its normals are minimal.

Let us introduce a function that sends a hyperplane to the sum of the entries of its minimal normal 
\[ n: \: H \rightarrow \sum_{i=1}^m (h_{min})_i.  \] 
Let us index the hyperplanes in $\zz^m_N$  by $\mathbb{N}$, in such a way that  if the index of $H$ is smaller than the index of $\tilde{H}$ it implies     $n(H) \leq n(\tilde{H})$. For example, the hyperplane defined by $h=(1,0)$ is $H_1$,  the one defined by $(0,1)$ is $H_2$, the one defined by $(1,1)$ is $H_3$, etc.. 
We are interested in evaluating the sum 
\[ S(m,N)=\sum_{i=1}^{N+1} n( H_i). \]
Note that this sum does not depend on the chosen indexing.

Let us denote by $c_{i}(m,N)$ the number of the hyperplanes $H$ in $\zz^m_N$ with $n(H)=i$.
Define 
\[ k_{m,N} \coloneqq \min  \left\{ k \in \mathbb{N} \: : \:   \sum_{i=1}^{k+1} c_{i}(m,N) > N+1  \right \}.  \]   
Now we can write 
\[ S(m,N) = \sum_{i=1}^{ k_{m,N}} i \:  c_{i}(m,N)   + \left( N+1 - \sum_{i=1}^{ k_{m,N}} c_{i}(m,N)\right) \left(k_{m,N}+1\right). \] 
In a similar manner we define 
\[ c_i(m) \coloneqq \Bigl| \Bigl\{ (x_1,\ldots,x_m) \in \zz_{\geq0}^m \: : \: \sum_{j=1}^m x_j = i \text{ and }
\gcd(x_1,\ldots,x_m)=1 \Bigr\} \Bigr|, \] 
\[ k_{m,N}^{\zz} \coloneqq \min  \left\{ k \in \mathbb{N} \: : \:   \sum_{i=1}^{k+1} c_i(m) > N+1  \right \}  \] 
and 
\[ S^{\zz}(m,N) \coloneqq \sum_{i=1}^{k_{m,N}^{\zz}} i \: c_i(m)  + \left( N+1 - \sum_{i=1}^{k_{m,N}^{\zz}} c_i(m)\right) \left(k_{m,N}^{\zz}+1\right). \] 
Note that this sum also makes sense for a non-prime $N$. 

\begin{lem}
    For a prime $N$ we have $S(m,N)  \geq  S^{\zz}(m,N)$.
\end{lem}
\begin{proof}
    
First of all, notice that $ c_i(m) \geq c_i(m,N)$. This implies that for a given $N$ the value of
$ k_{m,N}^{\zz} $ cannot be larger than $ k_{m,N}$.  Suppose $ k_{m,N}^{\zz} = k_{m,N} - x$ for some $x \geq 0$. 
Consider 
\begin{align*} S(m,N) -  S^{\zz}(m,N) = \sum_{i=1}^{k_{m,N}-x} \left( i - \left(k_{m,N}+1\right)\right) \left(c_i(m,N) - c_i(m)\right) -\\ - \sum_{i=k_{m,N}-x+1}^{k_{m,N}} c_i(m,N) \left(k_{m,N}+1-i\right) + x\left( \left(N+1\right) - \sum_{i=1}^{k_{m,N}-x}c_i(m)\right). \end{align*}
In the second sum $k_{m,N}+1-i$ is always positive and $k_{m,N}+1-i\leq x$, therefore if we bound $k_{m,N}+1-i$ with $x$ and use $c_i(m) \geq c_i(m,N)$ we get 
\begin{align*}
    &S(m,N) -  S^{\zz}(m,N) \geq \\ &\sum_{i=1}^{k_{m,N}-x} \left( i - \left(k_{m,N}+1\right)\right) \left(c_i(m,N) - c_i(m)\right)  + x \left( \left(N+1\right) - \sum_{i=1}^{k_{m,N}} c_i(m)\right).
\end{align*}
All the summands above are non-negative, thus $S(m,N) \geq  S^{\zz}(m,N)$. 
\end{proof}

Now we want to show that for each $m\geq2$ there exists $N_m^{\zz}$ such that for $N > N_m^{\zz}$ we have $S^{\zz}(m,N) > m  N$. This would imply that for a given $m$ there cannot be any spanning thin simplices with $N > N_m^{\zz}$.

\begin{lem}
The function $S^{\zz}(m,N)$ grows with $N$, namely 
\[  S^{\zz}(m,N+1) - S^{\zz}(m,N)=k^{\zz}_{m,N}+1. \]
 Moreover, this growth is non-decreasing. 
\end{lem}
\begin{proof}
    Note that the consecutive values of $k_{m,N}^{\zz}$ can either be the same or differ by one, namely \[k_{m,N+1}^{\zz}=k_{m,N}^{\zz}+\varepsilon_{m,N}\] with $\varepsilon_{m,N} \in \{0,1\}$.
Consider 
\begin{align*}
    S^{\zz}(m,N+1) - S^{\zz}(m,N)=k^{\zz}_{m,N}+1+ \varepsilon_{m,N} \left( N+2 - \sum_{i=1}^{k^{\zz}_{m,N}+\varepsilon_{m,N}} c_i(m) \right).
\end{align*}
When $\varepsilon_{m,N}=0$ the last summand does not appear. When $\varepsilon_{m,N}=1$ it means that we actually have  exactly $\sum_{i=1}^{k^{\zz}_{m,N}+1} c_i(m) = N+2$ and overall it gives 
\[ S^{\zz}(m,N+1) - S^{\zz}(m,N)=k^{\zz}_{m,N}+1 \]
This is always positive, so $S^{\zz}(m,N)$ always grows with increasing $N$. Moreover, this growth is non-decreasing since \[ \left( S^{\zz}(m,N+2) - S^{\zz}(m,N+1) \right) - \left( S^{\zz}(m,N+1) - S^{\zz}(m,N) \right) \] can only take values $0$ and $1$.  
\end{proof}
This lemma shows that there exists  $N_m^{\zz}$  such that  for every $N >N_m^{\zz}$ the difference $\left( S^{\zz}(m,N+1) - S^{\zz}(m,N) \right) \geq m+1$. We can find $N_m^{\zz}$ computationally for small values of $m$. 
By considering for each $N_m^{\zz}$ the closest prime from below we get the values of $N_m$ for $S(m,N)$ from Table \ref{tab:Nm}.
\end{proof}

\begin{rem}
	The number $c_i(m)$ is  also known as the number of new colors that can be mixed with $i$ units of $m$ given colors, see  \cite{oeis_foundation_inc_entry_2024} for $m=3$.
	For $m=2$ it is exactly the value of the Euler's totient function $c_i(2)=\phi(i)$. 
	The generating function of $c_i(m)$ can be given by 
	\[ \sum_{i \geq 0} c_i(m) t^i = \sum_{k \geq 1} \frac{\mu(k)}{(1-t^k)^m} . \]
	\end{rem}

\subsection{Lattice width} \label{subsec:width} In this subsection we show that lattice simplices of width $1$ have a useful description in terms of the corresponding linear codes. 

 The lattice width of a lattice polytope $\Delta$ is defined as the minimum of 
\[ \max_{x \in \Delta}  ~u(x) -  \min_{x \in \Delta} ~ u(x)\]
over all non-zero integer  linear forms $u$. 

The lattice polytopes of width $1$ are also known as \textit{Cayley polytopes}.  Another equivalent definition is as follows. A  $d$-dimensional lattice polytope $\Delta$ is Cayley of length $m$, if there exists a lattice projection $\zz^d \rightarrow \zz^{m-1}$ that maps $\Delta$ onto the standard simplex $\Delta_{m-1}$.  

A few similar notions were considered by Arnau Padrol in his PhD thesis \cite{padrol2013neighborly}.  In the first one,  instead of considering the standard simplex as the image of projection, one can relax this condition and project to any simplex.  One calls a point configuration \textit{affine Cayley} of length $m$  if there exists a  projection  $\mathbb{R}^d \rightarrow \mathbb{R}^{m-1}$ that maps $\Delta$ onto the vertex set of an $(m-1)$-dimensional simplex.  

For the second notion consider the following. Let $V$ be a vector configuration. It is called \textit{affine $\text{Cayley}^\star$} of length $m$ if there exists a partition 
 \[  V = V_1 \sqcup V_2 \sqcup \ldots \sqcup V_m, \] 
 such that $\sum_{v \in V_i} v =0 $ for all $i=1, \ldots, m$.

At first, the above two definitions seem not to be related. To connect the two  one has to introduce the concept of Gale duality.    Suppose $A$ is a $d$-dimensional point configuration consisting of $n$ vectors.  Let $M$ be the $(d+1) \times n$ matrix whose columns are the coordinates of the points from $A$ with $1$ appended, that is the homogeneous coordinates of the points of $A$.  Let $b_1, \ldots, b_{n-d-1}$ be a basis of the  kernel of $M$.  Set $M^\star$ to be the matrix whose rows are exactly the vectors $b_i$.  We define the vector configuration consisting of the column vectors of $M^\star$  to be a \textit{Gale dual} $A^\star$ of $A$.   Now we can state the following. 
 
\begin{prop}[Proposition 7.22 in \cite{padrol2013neighborly}]
A point configuration is affine Cayley if and only if its Gale dual is affine $\text{Cayley}^\star$. 
\label{prop:Padrol}
\end{prop}

We are interested in the situations when the point configuration $A$ is the configuration of the lattice points of a  simplex $\Delta$.  In this case we can read the matrix $M^\star$ from the linear code $C_\Delta$. 
Suppose $\Delta$ has  $d+1+m$ lattice points.  Then there are $m$ words $c_1, \ldots, c_m$ of height $1$ in $C_\Delta$. Each such word gives a row in the matrix $M^\star$ that takes the form 
\[  M^\star =  \begin{pmatrix}
	-N_\Delta & 0 & \ldots &0 & c_{10} & c_{11} & \ldots & c_{1 d} \\ 
		0& -N_\Delta  & \ldots& 0 & c_{20} & c_{21} & \ldots & c_{2 d} \\ 
		\vdots & \vdots & \vdots & \vdots&  	\vdots & \vdots & \vdots & \vdots \\ 
			0&0  & \ldots& -N_\Delta & c_{m0} & c_{m1} & \ldots & c_{m d} \\ 
\end{pmatrix}.\]
The last $d+1$ columns correspond to the vertices of the simplex and the first $m$ columns correspond to the $m$ remaining lattice points. 
Clearly, the rows are linearly independent, so the columns of this matrix indeed give us a Gale dual. 

For linear codes there exists a natural analog of affine $\text{Cayley}^\star$ configurations. 
\begin{dfn}
 We say that an extended linear code  of length $d+1$ over $\zz_N$ \textit{splits} into $m$ parts  if there is a set partition $I_1, \ldots, I_m$   of $\{0, 1, \ldots, d\}$  such that for any $j =1, \ldots, m$ and any word $c$ we have $ \sum_{i \in I_j} c_i = 0 \mod N$. 
\end{dfn}

This definition, in particular, implies that if the code $C$ splits, then for each height $1$ word $c$  there exists $k$, such that $c_i =0$ for all $ i \in \{0, \ldots ,d\} \setminus I_k$. This in turn implies that 
the above  matrix $M^\star$ has a nice block-diagonal structure after permuting the columns. Therefore, if $C$ splits, then the corresponding simplex is affine $\text{Cayley}^\star$. Moreover, the condition of splitting is stronger than just being affine $\text{Cayley}^\star$ and  it implies that the simplex must be Cayley, which we show in the next proposition. 

\begin{prop} \label{prop: Cayley}
A lattice simplex $\Delta$ is Cayley of length $m$ if and only if the corresponding linear code  splits into $m$ parts. 
\end{prop}
\begin{proof}
We will treat the case $m=2$. One can easily generalize it  to any $m$.  If $\Delta$ is Cayley it means, in particular, that we can find an affine unimodular transformation of $\zz^d$ such that there is a coordinate where the set of vertices of $I_1$ has a $0$ and the set of vertices $I_2$ has $1$.   Now it follows by construction that any word $c \in C_\Delta$ must have $ \sum_{i \in I_2} c_i =0 \mod N_\Delta$ and thus also $ \sum_{i \in I_1} c_i =0 \mod N_\Delta$. Therefore, the corresponding linear code  splits into two pieces.

Now suppose that the linear code splits. As we saw above, it implies that the simplex $\Delta$ is affine $\text{Cayley}$. 
Thus, after an affine unimodular transformation the set of the vertices of $\Delta$ splits into two parts $I_1$ and $I_2$. By Proposition \ref{prop:Padrol} we can write the coordinates of the vertices as follows. In the first part they have the form $v_i=(0, w_i)$ for $i \in I_1$ and  for some $w_i \in \zz^{d-1}$ and in the second part they are  of the form $v_i=(k, w_i)$ for $i \in I_2$ and  for some $k \in \mathbb{N}$. In other words, there is a projection onto the vertices of the $1$-dimensional simplex $[0,k]$.  

We are going to show that $k$ must equal $1$.
Consider two simplices $\Delta_1$ with vertices  $(0, w_i)$ for $i \in I_1$ and $(1, w_i)$ for $i \in I_2$ and $\Delta_k$ for some $k \geq 2$ with vertices  $(0, w_i)$ for $i \in I_1$ and $(k, w_i)$ for $i \in I_2$. Let $C_1$ and $C_k$ be the corresponding linear codes. Since $\Delta_1$ is Cayley, we know that the corresponding linear code splits.  Its words are defined by the rational numbers $\l_i^{(1)}$ such that
\[   \sum_{i \in I_1} \l_i^{(1)} (0, w_i,1)   + \sum_{i \in I_2} \l_i^{(1)} (1, w_i,1)   \in M \oplus \zz.   \]
The code $C_k$ is defined by the rational numbers 
$\l_i^{(k)}$ such that
\[   \sum_{i \in I_1} \l_i^{(k)} (0, w_i,1)   + \sum_{i \in I_2} \l_i^{(k)} (k, w_i,1)   \in M \oplus \zz .  \]
Note that  $N_{\Delta_1} \mid N_{\Delta_k}$.   All the tuples $(\l_0^{(1)}, \ldots, \l_d^{(1)})$ also give a tuple for the code $C_k$, therefore, 
 \[   \frac{N_{\Delta_k}}{N_{\Delta_1}} C_1  \subseteq C_k \]
 as sets. Suppose there are  tuples $(\l_0^{(k)}, \ldots, \l_d^{(k)})$ not coming from $C_1$, then necessarily $\sum_{i \in I_2} \l_i^{(k)} \notin \zz$ since otherwise it would correspond to some word of $C_1$.  This tuple then corresponds to a word that does not satisfy $ \sum_{i \in I_2} c_i =0 \mod N_{\Delta_k}$, i.e. the linear code $C_k$ does not split, a contradiction. Therefore, $k$ must be equal to $1$ and the corresponding simplex $\Delta$ is Cayley.  
\end{proof}

\section{Four-dimensional thin simplices} \label{sect: 4d}
\subsection{General strategy of classification}
This section is devoted to the proof of the classification of thin four-dimensional simplices presented in Theorem \ref{main theorem}. 
In this subsection we outline the main steps that lead to this result.   
From the previous section, we know that the search for thin simplices  is equivalent to  the search of linear codes over $\zz_N$ without maximal weight words or to the search of hyperplane arrangements  over $\zz_N$  with empty complement.  This way instead of working with $d$-dimensional simplices we can work simply with the generating matrices of linear codes. Let $C$ be a linear code over $\zz_N$ of length $d+1$ and let $g$ be an $m \times (d+1)$ matrix that generates $C$. We require this code to be extended and require that the greatest common divisor of all the entries of $g$ and $N$ is $1$, we write $\gcd(g,N)=1$.  Since the code is extended, the submodule $C$ is in fact a subgroup of $\zz_N^d$, so it is enough to consider $1 \leq m \leq d$.

\textbf{From now on let us fix the dimension to $d=4$}. From Proposition \ref{prop:gcdsum} it follows that if all the $\gcdvar_i$ satisfy $\gcdvar_i \leq N/5$, then the linear code $C$ cannot be thin.  Since each $\gcdvar_i$ is a divisor of $N$, they must be of the form $N/\a$ for $\a$ an integer. This leads to the following lemma. 
\begin{lem}
    If $C$ is thin, then at least one of the $\gcdvar_i$ is $N/\a$ with $\a \in \{2,3,4\}$.
\end{lem}
\noindent 
For a fixed  $d$ Corollaries \ref{cor:dimN} and \ref{cor : smallest prime} suggest that we might have more thin simplices when $N$ is small.  To embark upon the classification and simplify  the proof we classify the thin linear codes with $N \leq 8$ using a computer algebra system.  The following proposition is the outcome of a programme written in SageMath available online \footnote{\href{https://doi.org/10.5281/zenodo.18016483}{https://doi.org/10.5281/zenodo.18016483}}. 
\begin{prop} \label{prop3}
    For $N \leq 8$ there are  $10$ non-isomorphic thin linear codes that are not direct sums. Six of them are presented in the Table  \ref{sporadictable} and the remaining four are members of the family  \eqref{family1}.
\end{prop}



Building on this classification     of thin linear codes with $N \leq 8$, we will gradually cover all the possible remaining cases. We will start with $m=1$, i.e. generating matrices with one row, and proceed to $m=4$. It is helpful to note that if the linear code generated by $g$ is thin, then also the linear codes generated by subsets of the rows of $g$ are thin.  Moreover, each row of $g$ must contain at least one zero, otherwise we would immediately get a word of maximal weight. 

The case when the generating matrix has only one row is quite trivial, since this row must have a zero. 
\begin{lem}
If $m=1$ and $C$ is thin, then it corresponds to a lattice pyramid. 
\end{lem}

For the $m=2$ case we will proceed as follows. 
As we noted before, at least one of the columns of $g$ must have $\gcdvar_i = N/\a$ with $\a \in \{2,3,4\}$. If all the five columns have the same $\gcdvar_i$, then  $\gcd(g,N) \neq 1$. The same happens when four out of the five columns have the same $\gcdvar_i$ because $C$ is extended.  Let  \[ M_\a \coloneqq N/\a \]  be the maximal $\gcdvar_i$ of $g$, then we have to consider the cases when there are $1, 2$ or $3$ columns with $\gcdvar_i=M_\a$. They are covered in Subsection \ref{subsec:m=2} in the Lemmas \ref{lem:m2col1}, \ref{lem:m2col2} and \ref{lem:m2col3}, respectively. 
The outcome of these lemmas can be summarized in the following proposition.

\begin{prop} \label{prop4}
	For $m=2$ and $N\geq 9$  a non-degenerate linear code $C$ that is not a direct sum is thin if and only if $N$ is even and the generating matrix can be chosen as
  \be g = \begin{pmatrix}
	M_2 & M_2 & 0 & 0 & 0 \\
	M_2 & 0 & M_2 & 1 & N-1
\end{pmatrix}. \label{familymatrix}  \ee 
\end{prop}

Building on the $m=2$ case we can deal with the $m=3$ case using the fact that each pair of the rows of $g$ must generate a thin linear code, thus each such pair is either a multiple of a generating matrix from the $m=2$ situations or has a column of zeros.  In a similar manner we can treat the $m=4$ case.  In Subsections \ref{subsec:m=3} and \ref{subsec:m=4} we will prove the following.
\begin{restatable}{prop}{propfive} \label{prop5}
        Suppose $N\geq 9$, $g$ has three or four rows and it generates a non-degenerate linear code $C$ that is not a direct sum. If $C$ is thin, then it can be generated by a matrix with $2$ rows of the form \eqref{familymatrix}, i.e. there are no new interesting linear codes compared to the situation of $m=2$.  
\end{restatable}

All together, Propositions \ref{prop3}, \ref{prop4} and \ref{prop5} combine to Theorem \ref{main theorem} from Introduction.

\textit{Note that from now on the equality sign will predominantly mean equality mod N. In the few cases when confusion is possible we will write $=_{\zz}$ for the equality that has to be understood over $\zz$.}

\subsection{The $m=2$ case} \label{subsec:m=2}
In this part we are going to treat the generating matrices that have two rows.  
Recall that for $\a=2,3,4$
\[ M_\a \coloneqq  \frac{N}{\a}. \] 
\begin{lem} \label{lem:m2col1}
Let $N  \geq 9$.    Suppose $g$ has two rows and only one column of $g$ has $\gcdvar_i=M_\alpha$ and other columns have $\gcdvar_i<M_\a$. In this case $g$ does not generate a thin linear code, unless $g$ has a column of zeros. 
\end{lem}
\begin{proof}
\textit{No zeros in the column with $\gcdvar_i=M_\a$}. Suppose at first that the column with $\gcdvar_i=M_\a$ does not have zeros. Then since every row must have at least one zero, we can write 
\[ g = \begin{pNiceMatrix}[last-row]
    a_1 M_\a& b_1 & 0 & b_5 & b_7 \\ 
    a_2 M_\a  & 0 & b_4 & b_6 & b_8 \\
    H_0 & H_1 & H_2 & H_3 & H_4
\end{pNiceMatrix}\]
with $a_i \in \{0,1,\ldots,\a-1 \}$  such that $\gcd(a_1,a_2,\a)=1$ and $b_i \in \{0,1,\ldots, N-1 \}$. 
\textit{The indices $H_i$ of the columns denote the corresponding hyperplanes from $\mathcal{H}_C$ and they 
 are added for the reader's convenience.}
 
Consider a set of points 
\be \label{4pts} \left\{ (\a,1), \: (1,\a), \:(\a,-1), \: (-1,\a) \right\} \subseteq \zz_N^2. \ee
None of these points can be contained in $H_0, H_1$ or $H_2$, unless $\a b_1=0$ or $\a b_4=0$, but this would violate the assumption that only one column has $\gcdvar_i=M_\a$.  Thus, these points must be contained in the remaining two hyperplanes. Any triple of the above four points cannot lie in the same hyperplane without violating the assumptions of the lemma.  Therefore,  $H_3$ must contain a pair of the points and $H_4$ must contain the remaining pair. 

Suppose that $\a \neq 2$, then we might also consider the points
\[ \left\{ (\a,2), \:  \:(2,\a)\right\}. \]  
They are not in the hyperplanes $H_1$ or $H_2$. The hyperplane $H_0$ might contain one of these points if $\a=4$. Suppose that $(\a,2)$ is not in $H_0$. 
One can deduce that among the remaining two hyperplanes it can only be contained in the hyperplane that contains the points $( 1, \a)$ and $(-1, \a)$. Suppose it is $H_3$, then we must have $ 2 b_5=0$ and $ 4 b_6=0$. If $\a=4$, this would mean that $\gcdvar_3=M_4$, a contradiction.  Thus, we must have $\a=3$ and so   $(2,\a) \notin H_0$. This point may belong only to the hyperplane that contains $(\a,1)$ and $(\a,-1)$, that is $H_4$, and this corresponds  to $ 4 b_7=0$ and $ 2 b_8=0$.  These restrictions give us 
\[ g = \begin{pNiceMatrix}
    a_1 M_3& b_1 & 0 & N_\Delta/2 & \tilde{b}_7 N_\Delta/4 \\ 
    a_2 M_3  & 0 & b_4 & \tilde{b}_6 N_\Delta/4  & N_\Delta/2 
\end{pNiceMatrix}\]
with $ \tilde{b}_6,  \tilde{b}_7 \in \{1,3\}$. Consider the points $(1,1)$ and $(1,-1)$. One can check that now we cannot have both of these points in $\mathcal{H}_C$.

Consider now $\a=2$. Recall that each $H_3$ and $H_4$ must contain exactly a pair of points from \eqref{4pts}. Suppose $(\a,1), (1,\a) \in H_3$, then we must have
\[ g = \begin{pNiceMatrix}[last-row]
     M_2& b_1 & 0 & b_5 & b_7 \\ 
    M_2  & 0 & b_4 & b_6  & b_8 \\
    H_0 & H_1 & H_2 & H_3 & H_4
\end{pNiceMatrix}\]
with $3 b_5 = 3 b_6 = 3 b_7 = 3 b_8=0$. Consider the words corresponding to the points $(3,0)$ and $(0,3)$. They are $(M_2, 3 b_1,0, 0,0)$ and $(M_2, 0, 3 b_4, 0 ,0)$. Since $C$ is extended,  we have $ 3 b_1 =M_2$ and $ 3 b_4= M_2$. Therefore, the hyperplanes $H_1$ and $H_2$ do not contain all the points of the form $(3,x)$ and $(x,3)$ with non-zero $x$. In particular, the point $(3,2)$ has to be  either in $H_3$ or in $H_4$. It is easy to check that this is not possible in this situation.

If $(\a,1), (-1,\a) \in H_3$, the situation is similar. Now the corresponding conditions are $5 b_5 = 5 b_6 = 5 b_7 = 5 b_8=0$. It gives $ 5 b_1 =M_2$ and $ 5 b_4= M_2$. This again does not allow $H_1$ or $H_2$ to contain the point $(3,2)$. 
One can check that it also does not belong 
 to $H_3$ or $H_4$. 

If $(\a,1), (\a,-1) \in H_3$, then we have $ 4 b_5= 4 b_8 =0$ and $ 2 b_6 = 2 b_7=0$. It gives $ 4 b_1 =M_2$ and $ 4 b_4 = M_2$. Therefore, the point $(3,2)$ must be contained again in $H_3$ or $H_4$. This is not possible unless, for example, $b_5=0$ and $ 2 b_6=0$, but this would give $\gcdvar_3=M_2$. 

\textit{A zero in the column with $\gcdvar_i=M_\a$.}
Suppose the column with $\gcdvar_i=M_\a$ has a zero entry. Note that the points $(1,1)$ and $(1,-1)$ must be in $\mathcal{H}_C$, thus, the generating matrix must be of the form  
\[ g = \begin{pNiceMatrix}[last-row]
    a_1 M_\a& 0 & b_3 & b_5 & b_7 \\ 
   0  & b_2 & -b_3 & b_5 & b_8 \\
    H_0 & H_1 & H_2 & H_3 & H_4
\end{pNiceMatrix}.\]
Let us consider the following table, where the entries are the values (up to sign) on the corresponding points taken by the linear forms defining the hyperplanes $H_2$ and $H_3$.
\begin{table}[h]
  \centering
  \begin{tabular}{|c|c|c|c|c|c|c|}
    \hline
     & $(1,\a)$ & $(-1,\a)$ & $(\a+1,1)$ & $(\a+1,2)$ &$(\a+1,-1)$ & $(\a+1,-2)$ \\
    \hline
    $H_2$& $(\a-1)b_3$ & $(\a+1)b_3$  & $\a b_3$ & $ (\a-1) b_3$ & $(\a+2) b_3 $& $ (\a+3)b_3$ \\
    \hline
    $H_3$ &$(\a+1)b_5$ & $(\a-1)b_5$ & $(\a+2) b_5 $& $(\a+3) b_5$ &$ \a b_5$ &  $(\a-1) b_5$\\
    \hline
  \end{tabular}
  \label{tab:mytable}
\end{table}
None of the points from the table can be contained in $H_0$ or $H_1$.

Consider the pair of points $(\a+1,1)$ and $(\a+1,2)$.  None of them can be contained in $H_2$, since both of the conditions $\a b_3=0$ and $ (\a-1) b_3=0$ would imply $\gcdvar_i \geq M_\a$. Moreover, they cannot be both in $H_3$, since it would require $ b_5=0$. Thus, at least one of them must be in $H_4$.  The same logic applies to the pair of points $(\a+1, -1)$ and $(\a+1,-2)$. 

Suppose one of the points  $(\a+1,1)$ and $(\a+1,2)$ belongs to $H_3$, then the point $(1,\a)$ must be in $H_4$ since otherwise either $b_5=0$ or $2 b_5=0$.  On the contrary, if both of the points  $(\a+1,1)$ and $(\a+1,2)$ belong to $H_4$, then $(1,\a) \notin H_4$, since it would require $b_7=b_8=0$. Again, the same applies for the pair  $(\a+1, -1)$, $(\a+1,-2)$ and the point $(-1,\a)$. 

Therefore,  one of the following options must be satisfied: 
\begin{itemize}
    \item $(\a+1,1), \: (\a+1,-1),\: (1,\a), \: (-1, \a) \in H_4$,
    \item $(\a+1,1), \: (\a+1,-2),\: (1,\a), \: (-1, \a) \in H_4$,
    \item $(\a+1,2), \: (\a+1,-1),\: (1,\a), \: (-1, \a) \in H_4$,
    \item $(\a+1,2), \: (\a+1,-2),\: (1,\a), \: (-1, \a) \in H_4$,
    \item $(\a+1,1), \: (\a+1,2),\:(\a+1,-1), \: (\a+1,-2) \in H_4$.
\end{itemize}
It is a quick check that in the first four cases one obtains $2 b_7 = 2 b_8=0$ and hence $\gcdvar_4=N/2$, a contradiction. The last case requires $ b_8 = (\a+1) b_7 =0$. Moreover, in this case both of the points $(1,\a)$ and $(-1,\a)$  must be in $H_2 \cup H_3$, that is $(\a+1) b_3 = (\a+1) b_5=0$.  Consider the first row of $g$ and multiply it with $(\a+1)$, since $C$ is extended we have  $ a_1 M_\a + (\a+1) (b_3+b_5+b_7)  = a_1 M_\a=0$, i.e. the first column is a zero column, so the code is degenerate. 
\end{proof}

\begin{lem} \label{lem:m2col2}
    Let $N \geq 9$. Suppose $g$ has two rows and exactly two columns have $\gcdvar_i=M_\a$ and the other columns have $\gcdvar_i<M_\a$. In this situation $g$ cannot generate a thin non-degenerate linear code.  
\end{lem}
\begin{proof}
\textit{No zeros in the columns with $\gcdvar_i=M_\a$}.
    Suppose first that the columns with $\gcdvar_i=M_\a$ don't have any zeros, then 
     \NiceMatrixOptions
  {
    code-for-last-row = \scriptstyle , 
    code-for-first-col = \scriptstyle , 
  }
    \[ g = \begin{pNiceMatrix}[last-row]
     a_1 M_\a& a_3 M_\a & 0 & b_3 & b_5 \\ 
    a_2 M_\a  &  a_4M_\a & b_2 & 0 & b_6 \\
    H_0 & H_1 & H_2 & H_3 & H_4
\end{pNiceMatrix}.\]  
As in the previous lemma consider the points 
\[ \left\{ (\a,1), \: (1,\a), \:(\a,-1), \: (-1,\a) \right\}. \]
Under the assumptions none of these points can be contained in the first four hyperplanes. They cannot all be contained in $H_4$ as well. Therefore, the complement of $\mathcal{H}_C$ is not empty.  

\textit{A zero in one of the columns with $\gcdvar_i=M_\a$}.
    Suppose that one of the columns with $\gcdvar_i=M_\a$ has a zero, hence we can write
 \[ g= \begin{pNiceMatrix}
     a_1 M_\a& 0 & b_1 & b_3 & b_5 \\ 
    a_2 M_\a  &  a_4M_\a & 0 & b_4 & b_6 \\

\end{pNiceMatrix}.\] 

Suppose the points $(1,1)$ and $(1,-1)$ do not belong to $H_0$, then we must have 
 \[ g = \begin{pNiceMatrix}[last-row]
     a_1 M_\a& 0 & b_1 & b_3 & b_5 \\ 
    a_2 M_\a  &  a_4M_\a & 0 & -b_3 & b_5 \\
    H_0 & H_1 & H_2 & H_3 & H_4
\end{pNiceMatrix}.\]
Consider the points $(\a,1)$ and $(\a,\a+1)$.  They cannot be in $H_3$ since it would correspond to $(\a-1) b_3=0$ (violates that $\gcdvar_3<M_\a$) or $b_3=0$ (gives a column of zeros). Therefore, both of them must be in $H_4$, but this would give $\a b_5=0$, that is, $\gcdvar_4=M_\a$. 

Suppose now that the point $(1,1)$ is not in $H_0$, but $(1,-1)$ is.  Note that in this case the hyperplane containing $(1,1)$ cannot contain any of $(\a,1)$ and $(\a,\a+1)$. Therefore, both of these points must lie in the remaining hyperplane $H_4$, so we can write 
 \[ g_ = \begin{pNiceMatrix}
     a_1 M_\a& 0 & b_1 & b_3 & b_5 \\ 
    a_1 M_\a  &  a_4M_\a & 0 & -b_3 & -\a b_5
\end{pNiceMatrix}\]
with $\a^2 b_5=0$. 
Consider the word corresponding to the point $(0,\a)$
\[ (0, 0, 0, -\a b_3, 0). \]
Since the linear code must be extended, it is necessary to have $\a b_3 =0$ and $\gcdvar_3=M_\a$.  

Suppose that the point $(1,-1)$ is not in $H_0$, but $(1,1)$ is. The hyperplane containing $(1,-1)$ cannot contain $(\a,-1)$, otherwise it would have $\gcdvar_i>M_\a$. Therefore, we can write 
\[ g = \begin{pNiceMatrix}[last-row]
     a_1 M_\a& 0 & b_1 & b_3 & b_5 \\ 
    -a_1 M_\a  &  a_4M_\a & 0 & b_3 & \a b_5\\
    H_0 & H_1 & H_2 & H_3 & H_4
\end{pNiceMatrix}. \]
Since $(1,1) \in H_0$, $(1,-1) \notin H_0$  and the zeroth column has no zero entries, we must have $\a \neq 2$.  So we can consider the point $(2,-1)$. Clearly, it is not contained in the first 4 hyperplanes, and the last one can contain it only if $(\a-2) b_5=0$. If $\a=3$, then $b_5=0$, and the last column is a column of zeros. If $\a=4$, then $ 2 b_5=0$ and the last column is $( b_5, 0)^t$, but then $\gcdvar_4=M_2$.

We have deduced that both $(1,1)$ and $(1,-1)$ must belong to $H_0$. This is possible only if $\a=2$ and $a_1=a_2=1$. Using the fact that $C$ is extended we have 
\[ g = \begin{pNiceMatrix}[last-row]
     M_2 & 0 & b_1 & b_3 & M_2-b_1-b_3 \\ 
    M_2  &  M_2 & 0 & b_4 & -b_4\\
    H_0 & H_1 & H_2 & H_3 & H_4
\end{pNiceMatrix}. \]
Consider the  points $(\a,1)$ and $(2\a,1)$. If both belonged to $H_3$ we would have $ \a b_3=0$ and $b_4=0$, that is, $\gcdvar_3=M_2$.  The same is true for the pair of points $( \a,-1)$ and $( 2 \a, -1)$.  We have to distribute the points in these pairs between $H_3$ and $H_4$. It is enough to consider only two cases out of four.

Suppose $(\a,1)$ and $( 2 \a, -1)$ belong to $H_3$. The other two points must be in $H_4$. It gives 
\[  \a b_3 +b_4= 2 \a b_3 - b_4 = -\a (b_1+b_3)+b_4= - 2 \a (b_1+b_3)-b_4=0. \]
From this we deduce that $ 3 b_4=0$ and $ 3 \a b_3 = 3 \a b_1 =0$. Under this conditions none of $H_3$ and $H_4$ can contain the point $(\a, \a+1)$ unless $\a b_4=0$. Since $\a=2$,  this is equivalent to $ 2 b_4=0$, which together with $3 b_4=0$ gives $b_4=0$. Thin, in turn, together with $ \a b_3 + b_4=0$ implies $ 2 b_3=0$, that is $\gcdvar_3=M_2$.

Now suppose $(\a,1)$ and $( \a, -1)$ belong to $H_3$.
It gives 
\[  \a b_3 +b_4=  \a b_3 - b_4 = -2 \a (b_1+b_3)-b_4= - 2 \a (b_1+b_3)+b_4=0. \]
It follows that $ 2 b_4=0$. We cannot have $b_4=0$ since it would give $\gcdvar_3=M_\a$, therefore we need $\a=2$ and $ b_4= M_2$. Moreover, we must have $ 2\a b_3 = 4 b_3 =0$ and $ 4 \a b_1 =8 b_1 =0$. Therefore, we can write 
\[ g = \begin{pNiceMatrix}
     M_2 & 0 & \tilde{b}_1 N/8 & \tilde{b}_3 N/4 & M_2-\tilde{b}_1 N/8- \tilde{b}_3 N/4  \\ 
    M_2  &  M_2 & 0 & M_2 & M_2
\end{pNiceMatrix}. \]
for $\Tilde{b}_1 \in \{1,2, \dots, 7\}$ and $\tilde{b}_3 \in \{1,3\}$.  We see that $N/8 \mid g$ and we can reduce to the existing classification of linear codes with $N\leq 8$.  

\textit{zeros in the columns with $\gcdvar_i=M_\a$}. 
Suppose that both columns with $\gcdvar_i=M_\a$ have a zero. Suppose at first that these zeros are in the different rows.  Since the first two hyperplanes do not contain $(1,1)$ and $(1,-1)$ we can write 
 \[ g = \begin{pNiceMatrix}[last-row]
     a_1 M_\a& 0 & b_1 & b_3 & b_5 \\ 
    0 &  a_4M_\a & -b_1 & b_3 & b_6 \\
    H_0 & H_1 & H_2 & H_3 & H_4
\end{pNiceMatrix}.\] 
Assume that $\a \neq 2$. Consider the points $(2,1), \: (1,2),\: (2,-1), \: (-1,2)$. One can check that none of them can belong to the first four hyperplanes under the given assumptions, thus, they must be in $H_4$, but this is possible only if it is a zero column. 

Now consider $\a=2$. Since the code is extended, $b_5= M_2 -b_1 -b_3$ and $b_6 = M_2 +b_1 - b_3$.
Consider the following table that contains points and the values (up to a sign) achieved on these points  by  the linear functionals defining $H_2,H_3$ and $H_4$.
\begin{table}[h]
  \centering
  \begin{tabular}{|c|c|c|c|c|c|c|}
    \hline
     & $(3,1)$ & $(3,-1)$ & $(1,3)$ & $(1,-3)$ &$(5,1)$ & $(5,-1)$ \\
    \hline
    $H_2$& $ 2 b_1$ & $ 4 b_1$  &  $2 b_1$ & $4 b_1$ & $ 4 b_1$& $ 6 b_1 $ \\
    \hline
    $H_3$ & $ 4 b_3$ & $ 2 b_3$ & $ 4 b_3 $ & $ 2 b_3 $ &$ 6 b_3$ & $ 4 b_3$\\
    \hline
    $H_4$ & $  2 b_1 + 4 b_3 $ & $  4 b_1 + 2 b_3$ & $2b_1 -4b_3$ & $4 b_1 - 2 b_3$ & $4 b_1 + 6 b_3$ & $  6 b_1 + 4 b_3$  \\
    \hline
  \end{tabular}
\end{table}

All of the above points will be in the hyperplane arrangement if $ 4 b_1 = 4 b_3=0$, but this would imply $N/4 \mid g$. If both $ 4 b_1 \neq 0$ and $ 4 b_3 \neq 0$ then some of these points are not in $\mathcal{H}_C$. Therefore, exactly one of $ 4 b_1$ and $ 4 b_3$ must be zero.  We can choose $ 4 b_1 =0$ and $ 4 b_3 \neq0$. The other case  clearly gives an equivalent linear code.  Now the point $(3,1)$ must belong to $H_4$, i.e. we must have $ 2 b_1 + 4 b_3 =0$, which implies $ 8 b_3=0$. Thus,  $ N/8 \mid g$ and we again reduced to the existing classification of codes with $N\leq 8$. 

Now suppose that both of the zeros from the columns with $\gcdvar_i=M_\a$ are in the same row, so we have 
 \[ g = \begin{pNiceMatrix}
     a_1 M_\a& a_3 M_\a & b_1 & b_3 & b_5 \\ 
    0 &  0 & -b_1 & b_3 & b_6 \\
\end{pNiceMatrix}.\] 
In the same way as in the previous case we deduce that we need $\a=2$, thus 
 \[ g = \begin{pNiceMatrix}
     M_2&  M_2 & b_1 & b_3 & -b_1-b_3 \\ 
    0 &  0 & -b_1 & b_3 & b_1-b_3 
\end{pNiceMatrix}.\] 
The last table remains unchanged for this situation, thus we reduce again to the situation $N \leq 8$.
\end{proof}

\begin{lem} \label{lem:m2col3}
   Let $N \geq 9$. Suppose $g$ has two rows and exactly three columns have $\gcdvar_i=M_\a$ and the other columns have $\gcdvar_i<M_\a$.  In this case for even $N$ there exists a thin linear code that is not a direct sum and it is generated by 
     \be g = \begin{pmatrix}
	M_2 & M_2 & 0 & 0 & 0 \\
	M_2 & 0 & M_2 & 1 & N-1
\end{pmatrix}. \label{lemmamatrix} \ee 
Note that these linear codes correspond to simplices of width $1$ due to Proposition \ref{prop: Cayley}.
\end{lem}
\begin{proof}
\textit{No zeros in the columns with $\gcdvar_i=M_\a$}.
    Suppose that the columns with $\gcdvar_i=M_\a$ don't have any zeros, then 
     \NiceMatrixOptions
  {
    code-for-last-row = \scriptstyle , 
    code-for-first-col = \scriptstyle , 
  }
    \[ g = \begin{pNiceMatrix}
     a_1 M_\a& a_3 M_\a & a_5 M_\a  & 0 & b_3 \\ 
    a_2 M_\a  &  a_4M_\a & a_6 M_\a  & b_2 & 0 \\
\end{pNiceMatrix}.\]  
In this case  $M_\a \mid g$ since the linear code is extended. 

\textit{At least one zero in the columns with $\gcdvar_i=M_\a$ but only in one row}.
Consider  \[ g = \begin{pNiceMatrix}[last-row]
     a_1 M_\a& a_3 M_\a & a_5 M_\a  & 0 & b_3 \\ 
    a_2 M_\a  &  a_4M_\a & 0  & b_2 & b_4 \\
    H_0 & H_1 & H_2 & H_3 & H_4
\end{pNiceMatrix}\] 
with possibly $a_2$ and $a_4$ being zero and other $a_i$ being nonzero. 
Since the linear code is extended, $b_3$ is divisible by $M_\a$.
Consider the points $(1,\a)$ and $(-1,\a)$. They cannot be contained in the first four hyperplanes, therefore, both of them belong to $H_4$. This leads to $ 2 \a b_4 =0$ and $ 2 b_3=0$, forcing $\a \neq 3$ since $M_\a \mid b_3$.  If $\a=2$, then $N/4$ divides $g$. So we need to consider only $\a=4$.

Consider the points $(2,1)$ and $ (3,1)$. Since $ 2 b_3 =0$  none of these points can be in $H_4$ unless it is a column of zeros or $\gcdvar_4=N/2$. 
Thus, they must belong to the hyperplanes $H_0$ and $H_1$.  Clearly, they cannot be in the same one of these hyperplanes.  We can assume that $(2,1) \in H_0$ and $(3,1) \in H_1$.  This implies that the generating matrix $g$ is divisible by $N/8$ so we can reduce to the case $N=8$. 

\textit{Two columns with $\gcdvar_i=M_\a$ have zeros in different rows.}
Consider
 \[ g = \begin{pNiceMatrix}[last-row]
     a_1 M_\a&0  & a_5 M_\a  & b_1 & b_3 \\ 
    a_2 M_\a  &  a_4M_\a & 0  & b_2 & b_4 \\
    H_0 & H_1 & H_2 & H_3 & H_4
\end{pNiceMatrix}\]  
with possibly one of $a_1$ or $a_2$ being zero.
Suppose at first that $\a \neq 2$.  The points $(1,1)$ and $(1,-1)$ cannot belong to the same hyperplane unless its $\gcdvar_i=M_2$, therefore, we have to distribute them among two different hyperplanes. Suppose that $(1,1) \in H_3$ and $(1,-1) \in H_4$.  This implies that the points $(2,1), \:(2,-1),\: (1,2)$ and $(1,-2)$ do not belong to neither of these hyperplanes, unless one of them is a zero column or they have $\gcdvar_3=\gcdvar_4=M_3$. Thus, all these points must belong to $H_0$, but it is possible only if it is a zero column. 

Another option is to have $(1,1) \in H_0$ and $(1,-1) \in H_3$ (the situation of $(1,-1) \in H_0$ and $(1,1) \in H_3$ is equivalent).  We have 
 \[ g = \begin{pNiceMatrix}[last-row]
     a_1 M_\a&0  & a_5 M_\a  & b_1 & b_3 \\ 
    -a_1 M_\a &  a_4M_\a & 0  & b_1 & b_4 \\
    H_0 & H_1 & H_2 & H_3 & H_4
\end{pNiceMatrix}.\] 
The points $(2,1)$ and $(1,2)$ cannot belong to $H_0$ and also they do not belong to $H_3$ since it would imply $\gcdvar_3=M_3$. The only possibility left is $(2,1), \: (1,2) \in H_4$, but this  implies $\gcdvar_4=M_3$ as well. 

Consider now $\a=2$. The generating matrix takes form
\[ g = \begin{pNiceMatrix}[last-row]
     a_1 M_2&0  &  M_2  & b_1 & (-a_1+1) M_2 - b_1 \\ 
    M_2 & M_2 & 0  & b_2 & -b_2 \\
    H_0 & H_1 & H_2 & H_3 & H_4
\end{pNiceMatrix}.\] 
Suppose $a_1=0$. Then without loss of generality $(1,1)$ must be in $H_3$. The point $(3,1)$ then must be in $H_4$, which would imply $ 4 b_1 =0$ and $N/4 \mid g$.  

So we have to consider only the case when $a_1=1$ and 
\be \label{family} g = \begin{pNiceMatrix}
     M_2&0  &  M_2  & b_1 &  - b_1 \\ 
    M_2 & M_2 & 0  & b_2 & -b_2 
\end{pNiceMatrix}.\ee
It is easy to see that for any choice of $b_1$ and $b_2$ these linear codes are thin.  The proof  is concluded by the following lemma

\begin{lem}
\label{familieslemma}
For any non-zero choice of $(b_1,b_2)$ the linear code $C$ generated by \eqref{family} is either a direct sum or it is isomorphic to the code generated by 
        \[ g = \begin{pmatrix}
            M_2 & M_2 & 0 & 0 & 0 \\
            M_2 & 0 & M_2 & 1 & N-1
        \end{pmatrix}. \] 
The corresponding simplex has 
\[h^*(\Delta_C,t) = \left(\frac{3 N}{2}  -1\right) t^2 + \frac{N}{2} t +1.\] 
\end{lem}

\begin{proof}
Let us describe all the words in the linear code generated by \eqref{family}.
First, let us show that all the words of the form \[ (0, 0,0, 2 k, -2 k) \]
 for $k\in\{0,1,\ldots, N/2-1\}$ appear in $C$. For this it is enough to show that there are coefficients $(c_1, c_2)$ that give the word $(0,0,0,2, N-2)$. If $\gcd(b_1, b_2, M_2) \neq 1$, then it divides $g$  and we reduce to the same type of situation but for a lower $N$. Therefore, we can assume that $\gcd(b_1, b_2, M_2)=1$. By Bezout's identity there exist integers $x,y,z$ such that $ x b_1 + y b_2 +z M =_{\zz}1$ consequently $ 2 x b_1 + 2 y b_2 + z N =_{\zz}2$. Considering this identity modulo $N$ we see that we can take $ (c_1, c_2) = ( 2 x, 2 y) \mod N$.  Note that there are no words of the form $(0,0,0, k,-k)$ for odd $k$. 

Suppose, $b_1$ and $b_2$ are odd.  Since $\gcd(b_1, b_2, N)=1$ there are integers $x,y,z$ such that 
\[ x b_1 + y b_2 + z N=_{\zz}1. \]
Since $N$ is even and $b_1, b_2$ are odd, exactly one of  $x$ and $y$ must be even. 
Considering this identity modulo $N$ we deduce that there are such coefficients $(c_1,c_2)$ that give us words $( M_2, M_2, 0, 1, -1)$ and $(M_2,0, M_2, 1, -1)$. Multiplying these with odd numbers we get all the words  of the form
\[ (M_2, M_2,0, 2 k +1, -2 k -1) \quad \text{ and } \quad  (M_2, 0, M_2, 2 k +1, -2 k -1).\]

We are left to consider the coefficients $(c_1,c_2)$ with both $c_i$ odd.
Take $(c_1,c_2)=(b_2,-b_1)$. This gives us the word $(0, M_2, M_2, 0,0)$. Adding to this word all the words of the form $(0,0,0, 2 k, -2 k)$ from before we can get all the words of the form 
\[ (0,M_2, M_2, 2 k, -2k). \] 
This way we exhausted all the possible coefficients $(c_1, c_2)$. 

Now suppose that $b_1$ is odd and $b_2$ is even. The same considerations as above allow one to deduce that the corresponding linear code $C$ contains all the words of the form $(0, M_2, M_2, 2 k+1, -1 - 2k), \: (M_2, M_2, 0, 2k+1, -1 - 2 k), \: (M_2,0,M_2, 2k, -2k)$. So $C$ is not exactly the same as the codes with odd $b_1$ and $b_2$ but a  permutation of the first three columns gives an isomorphism between them. 

Finally, suppose that both $b_1$ and $b_2$ are even. This is possible only if $ 4 \nmid N$ since otherwise we would have $ 2 \mid g$. Since $4 \nmid N$, we have odd $M_2$. Consider the word $(M_2, M_2,0, b_1, -b_1)$ and multiply it with $M_2$. Since $b_1$ is even, it gives $(M_2, M_2, 0, 0, 0)$. In a similar way we obtain $(M_2,0, M_2, 0,0)$ from $(M_2,0, M_2, b_2, -b_2)$. We also have $(0, M_2, M_2, 0,0)$ as their sum. Now adding the words $(0,0,0, 2k, -2k)$ to these, we obtain all the possible words in $C$. Having the list of all the words in $C$,  we see that  in the case of $b_1$ and $b_2$ being even the linear code can be generated by the matrix 
\[ \begin{pmatrix}
    M_2 & M_2 & 0& 0&0 \\
    M_2 & 0 & M_2 & 0& 0 \\
    0 & 0 & 0 & 2 & -2
\end{pmatrix} \]
as well.  Thus, this  is a direct sum and the corresponding simplex is a free join. 

Since for different choices of $b_1$ and $b_2$ the codes are isomorphic, we can simply fix $(b_1,b_2)=(0,1)$ and arrive so at the statement of the lemma.  It is easy to compute the corresponding $h^*$-polynomials, since the words correspond to the integral points of the half-open parallelepiped. 
\end{proof}
\end{proof}
With this lemma we have covered all the possible cases of $g$ with two rows. Now we can move on to the generating matrices with three rows.

\subsection{$m=3$ cases} \label{subsec:m=3}

\setcounter{thm}{4}

\begin{prop}[Part 1] \label{m=3 1}
        Suppose $N\geq 9$, $g$ has three  rows and it generates a non-degenerate linear code $C$ that is not a direct sum. If $C$ is thin, then it can be generated by a matrix with $2$ rows of the form \eqref{familymatrix}. 
\end{prop}
\setcounter{prop}{9}

\begin{proof}
If a matrix $g$ with three rows generates a thin linear code, then matrices constructed from the pairs of rows of $g$ must also generate thin linear codes.  We can choose the first two rows to be a multiple of a generating matrix of a thin code and add a third row to it. For the first two rows we have six options: a matrix with a column of zeros, multiples of the cases 2,3,4,6 from Table \ref{sporadictable} and multiples of the matrices from Lemma \ref{familieslemma}.
We have to treat them one by one. 

\textbf{Case 2 from Table \ref{sporadictable}}. We look at the generating matrices over $\zz_{3 k}$ for $k\geq 3$ of the form 
\[  g = \begin{pmatrix}
    0 & 0 & k& k & k \\
    k & 2 k & 0 & k & 2 k  \\
    a_0 & a_1 & a_2 & a_3 & a_4 
\end{pmatrix} \] 
such that $\gcd(g)=1$. 
If we want the first and the third  rows to generate a thin code, then we need $a_0=0$ ($a_1=0$ is equivalent). Same argument applied to the second and third row gives $a_2=0$. We reduced to 
\[  g = \begin{pmatrix}
    0 & 0 & k& k & k \\
    k & 2 k & 0 & k & 2 k  \\
    0 & a_1 & 0 & a_3 & -a_1-a_3 
\end{pmatrix}. \] 
Consider the points $(1,1,1)$ and $(1,2,1)$.  One can check that they cannot be both in $\mathcal{H}_C$ unless $k \mid g$.  

\textbf{Cases 3,4 and 6 from Table \ref{sporadictable}}.  All these cases behave quite similarly in this situation. We present here only the case 3.   We look at the generating matrices over $\zz_{4 k}$ for $k\geq 3$ of the form 
\[  g = \begin{pmatrix}
    0 & 0 & k& k & 2 k \\
   2  k & 2 k & k & 3k & 0  \\
    a_0 & a_1 & a_2 & a_3 & a_4 
\end{pmatrix}. \] 
Since the first and the third row must generate a thin code, we need $a_0=0$ ($a_1=0$ is equivalent).  The second and the third row 
generate a thin code under one of the following conditions: 
\begin{enumerate}
    \item $a_4=0$,
    \item $a_1=2 k$, $a_4= 2k$. 
\end{enumerate}
In the first scenario we have 
\[  g = \begin{pmatrix}
    0 & 0 & k& k & 2 k \\
   2  k & 2 k & k & 3k & 0  \\
    0 & a_1 & a_2 & -a_1-a_2 & 0
\end{pmatrix}. \] 
Consider again the points $(1,1,1)$ and $(1,2,1)$. One can check, that we cannot cover both of them unless $k \mid g$.

In the second scenario we have 
\[  g = \begin{pmatrix}
    0 & 0 & k& k & 2 k \\
   2  k & 2 k & k & 3k & 0  \\
    0 & 2k & a_2 & -a_2 & 2k
\end{pmatrix}. \] 
The point $(1,1,2)$ can be in $\mathcal{H}_C$ if either $a_2=0$ or $2 a_2= 2 k$, but both of these conditions lead to $ k \mid g$. 



\textbf{First two rows given by a generating matrix belonging to the family \eqref{family1}}. 
Let $N$ be even. Consider generating matrices over $\zz_{ N}$ of the form 
\[ g = \begin{pmatrix}
    M_2 & M_2 & 0 & 0 & 0 \\
    M_2 & 0 &M_2 & 1 & -1\\
    a_0 & a_1 & a_2 & a_3 & a_4
\end{pmatrix}. \] 
One can  show that if we considered instead the first two rows of $g$ multiplied by an integer $k \geq 2$, then similar to the already treated situations  we would arrive at $ k \mid g$. Therefore, it is enough to consider the matrix above. 

First of all, we need that the the first and the third rows 
\[ \begin{pmatrix}
    M_2 & M_2 & 0 & 0 & 0 \\
    a_0 & a_1 & a_2 & a_3 & a_4
\end{pmatrix}\] 
generate a thin linear code. It is possible if any of the following conditions is true: 
\begin{enumerate}
    \item $a_2=0$,  \label{cond1}
    \item $a_3=0$, \label{cond2} 
    \item $a_2 = M_2$, $a_0=M_2$, $a_1=0$, \label{cond3}
    \item $a_2 = M_2$, $a_0=0$, $a_1=M_2$.\label{cond4}
\end{enumerate}
We also need the same for the second and third rows
\[\begin{pmatrix}
    M_2 & 0 &M_2 & 1 & - 1 \\
    a_0 & a_1 & a_2 & a_3 & a_4
\end{pmatrix}. \] 
It is possible if 
\begin{enumerate}
\setcounter{enumi}{4}
    \item  $a_1=0$, \label{cond5}
    \item  \label{cond6}$a_1=M_2$, $a_0=M_2$, $a_2=0$,
    \item \label{cond7} $a_2 = M_2$, $a_0=0$, $a_1=M_2$ (note, this is the same as condition (\ref{cond4})). 
\end{enumerate}

There are the following consistent pairs of the above conditions that do not trivially give a linear code equivalent to the one generated by the first two rows: (\ref{cond1}) and (\ref{cond5}), (\ref{cond1}) and (\ref{cond6}), (\ref{cond2}) and (\ref{cond5}), (\ref{cond3}) and (\ref{cond5}), (\ref{cond4}) and (\ref{cond7}).

Let us consider the above pairs one by one.

\textit{Conditions (\ref{cond1}) and (\ref{cond5})}.  In this case we have 
\[ g= \begin{pNiceMatrix}[last-row]
    M_2 & M_2 & 0 & 0 & 0 \\
    M_2 & 0 &M_2 & 1 & - 1 \\
    a_0 & 0 & 0 & a_3 & a_4\\
    H_0 & H_1 & H_2 & H_3 & H_4
\end{pNiceMatrix}.\]
If $(1,1,1) \in H_0$, then $a_0=0$ and $a_4=-a_3$. If $a_3$ is even, the the third row is a linear combination of the first two as we have seen in Lemma \ref{familieslemma}. If $a_3$ is odd, then we can show that the resulting linear code is equivalent to the one generated by 
\be \label{free} g= \begin{pmatrix}

    M & M & 0 & 0 & 0 \\
    M & 0 &M & 0 & 0 \\
    0 & 0 & 0 & 1& -1
\end{pmatrix}.\ee 
See Lemma \ref{lemma_FJ} below.

If $(1,1,1) \notin H_0$, then up to a permutation of the last two columns it must lie in $H_3$. We have 
\[ g= \begin{pNiceMatrix}[last-row]
    M _2& M_2 & 0 & 0 & 0 \\
    M_2 & 0 &M_2 & 1 & - 1 \\
    a_0 & 0 & 0 & -1 & 1-a_0 \\
    H_0 & H_1 & H_2 & H_3 & H_4
\end{pNiceMatrix}.\]
Consider the point $(1,1,-1)$. It cannot be in the hyperplane arrangement unless $a_0=0$ (which leads to the situation of Lemma \ref{lemma_FJ}) or $a_0= 2 $ in which case the point $(1,1,2)$ is neither in $H_0$ nor in $H_4$ as long as $N \geq 9$. 

\begin{lem}
    \label{lemma_FJ}
    The linear code over $\zz_N$ generated by 
    \[ \label{g1}g= \begin{pmatrix}
    M_2 & M_2 & 0 & 0 & 0 \\
    M_2 & 0 &M_2 & 1 & - 1 \\
    0 & 0 & 0 & a & -a
\end{pmatrix}\]
with  odd $a$ is isomorphic to the code generated by 
\[ \label{g2} g_0= \begin{pmatrix}
    M_2 & M_2 & 0 & 0 & 0 \\
    M_2 & 0 &M_2 & 0 & 0 \\
    0 & 0 & 0 & 1 & -1
\end{pmatrix}.\]
This is a direct sum, in particular, the corresponding simplex is a free join of $ 2 \Delta_2$ and a $1$-dimensional interval of length $N$.  
\end{lem}
\begin{proof}
   Since $a$ is odd, the following equalities hold \begin{align*}
    &(M_2,0, M_2, 0,0) = a \; (M_2, 0, M_2, 1, -1) - \; ( 0, 0, 0, a,-a), \\
    &(0,0,0,1,-1) = (1-a) \;(M_2, 0, M_2, 1,-1) +  \; (0 , 0, 0, a,-a) \end{align*}
   This is an invertible linear transformation between the last two rows of $g$  and $g_0$. 
Therefore,  the linear codes generated by $g$ and $g_0$ are isomorphic. 
\end{proof}

\textit{Conditions (\ref{cond1}) and (\ref{cond6}).} In this case the generating matrix is 
\[ g= \begin{pNiceMatrix}
    M_2 & M_2 & 0 & 0 & 0 \\
    M_2 & 0 &M_2 & 1 & - 1 \\
    M_2 & M_2 & 0 & a_3 & -a_3
\end{pNiceMatrix}.\]
We can substitute the third  row by the difference of the third and the first rows. 
Now again it is either the linear code generated by the first two rows if $a_3$ is even or it is the situation of the Lemma \ref{lemma_FJ} if $a_3$ is odd. 

\textit{Conditions (\ref{cond2}) and (\ref{cond5}).} In this case the generating matrix is 
\[ g= \begin{pNiceMatrix}[last-row]
M_2 & M_2 & 0 & 0 & 0 \\
    M_2 & 0 &M_2 & 1 & - 1 \\
a_0 & 0 & a_2 & 0 & -a_2-a_0\\
H_0 & H_1 & H_2 & H_3 & H_4
\end{pNiceMatrix}.\]
Consider the words corresponding to the points $(1,2,1)$ and $(1,2,-1)$. They are 
\[ (a_0+M_2, M_2, a_2 , 2  , - 2 - a_2 - a_0), \quad  (-a_0+M_2, M_2, -a_2 , 2  , - 2 + a_2 + a_0). \]
There are exactly three options how both of them can have a zero. 
The first one is $a_2=0$, i.e. \[ g= \begin{pNiceMatrix}
M_2 & M_2 & 0 & 0 & 0 \\
    M_2 & 0 &M_2 & 1 & - 1 \\
a_0 & 0 & 0 & 0 &-a_0\\
\end{pNiceMatrix}.\]
The point $(1,1,1)$ 
can be contained only in $H_4$, i.e. $a_0=-1$. Consider now the word corresponding to the point $(1,1,2)$. It is $( - 2, M_2, M_2,1,1)$, so the code generated by $g$ cannot be thin. 

The second option is to have $a_2 = -a_0 - 2  = -a_0 + 2$. This implies $4 =0$, so necessarily $N=4$. 

The only option left is to have $a_0=M_2$, which gives 
 \[ g= \begin{pNiceMatrix}
M_2 & M_2 & 0 & 0 & 0 \\
    M_2 & 0 &M_2 & 1 & - 1 \\
M_2 & 0 & a_2 & 0 &M_2-a_2\\
\end{pNiceMatrix}.\]
Consider  the points $(1,1,1)$ and $(1,1,-1)$. They correspond to the words 
\[ (M_2, M_2, M_2+a_2, 1, M_2 - 1 -a_2), \quad (M_2, M_2, M_2-a_2, 1, M_2 - 1 +a_2). \] 
There are two ways for both  of these words to have a zero. It is possible if  $ a_2 = M_2 -1 = M_2 +1$, but it implies $2 =0$, that is $N=2$.  The other option is $ a_2 = M_2$. 
Now we can substitute the second row with the difference of the second and the third rows and we see that  this linear code is a direct sum and the corresponding simplex is a free join of $ 2 \Delta_2$ and the interval of length $N$.

\textit{Conditions (\ref{cond3}) and (\ref{cond5}).}  In this case the generating matrix is
\[ g= \begin{pNiceMatrix}
    M_2 & M_2 & 0 & 0 & 0 \\
    M_2 & 0 &M_2 & 1 & - 1 \\
    M_2 & 0 & M_2 & a_3 & -a_3\\
\end{pNiceMatrix}.\]
If $a_3$ is odd, then the last row is already a word in the linear code generated by the first two rows, so we need to consider only even $a_3$. In this case from the last two rows we can get the words \begin{align*} &(M,0,M,0,0) =  a_3 \: ( M, 0, M, 1, -1) - (M, 0, M, a_3, -a_3), \\ 
 &(0,0,0,1,-1) = ( a_3+1) \: (M,0,M,1,-1) - (M,0,M, a_3, -a_3). \end{align*} 
 This is an invertible transformation, so  the generating matrix 
\[ \begin{pNiceMatrix}
    M_2 & M_2 & 0 & 0 & 0 \\
    M_2 & 0 &M_2 & 0 & 0 \\
    0 & 0 & 0 & 1 & -1\\
\end{pNiceMatrix}\]
generates the same linear code.

\textit{Conditions (\ref{cond4}) and (\ref{cond7}).}  In this case the generating matrix is
\[ g= \begin{pNiceMatrix}
    M_2 & M_2 & 0 & 0& 0 \\
    M_2 & 0 &M_2 & 1&  -1 \\
    0 & M_2 & M_2 & a_3 & -a_3
\end{pNiceMatrix}.\]
Again we have either a direct sum or the third row is a combination of the first two. The proof is almost identical to the previous case.

\textbf{A column of zeros in every pair of rows }
So far we have showed that if any pair of the rows of a generating matrix with three rows is non-degenerate, then there are no new interesting linear codes comparing to the case of just two rows. 
The only option left to consider is when all the restrictions from three rows to two rows have a column of zeros.
In this case the generating matrix takes the form
     \[ g=\begin{pNiceMatrix}[last-row]
    d_1 & 0 & 0 & b_1 & b_4 \\
    0& d_2 & 0& b_2 & b_5 \\
    0 & 0& d_3 & b_3 & b_6  \\ 
H_0 & H_1 & H_2 & H_3 & H_4
\end{pNiceMatrix}. \]
We will show that this matrix cannot generate a thin linear code for $N\geq 9$ unless this code is a direct sum. 

    Consider the points $(1,1,1), \: (1,1,-1), \:(1,-1,1),\:(-1,1,1)$. They must be contained in the union  $H_3 \cup H_4$. One can check that if one of these hyperplanes contains three out of the four points, then it necessarily contains the fourth one as well. Thus, 
    up to permuting these hyperplanes there are two possible situations: either $H_3$ contains  two out of four points or all of them.

    Let us start with the case when all the four points are in $H_3$. It requires $ 2 b_1 = 2 b_2 = 2 b_3 =0$ and $ b_1 + b_2 + b_3=0$. We can choose $ b_1=b_2=M_2$ and $b_3=0$.  Now we have
     \[ g=\begin{pNiceMatrix}[last-row]
    d_1 & 0 & 0 & M_2& M_2-d_1\\
    0& d_2 & 0& M_2 & M_2-d_2 \\
    0 & 0& d_3 & 0 & -d_3  \\ 
H_0 & H_1 & H_2 & H_3 & H_4
\end{pNiceMatrix}. \]
Consider the points of the form $(2,\pm1,\pm1)$. They can be either in $H_0$ or $H_4$. If $ 2 d_1 \neq 0$, then all of them must be in $H_4$. In that case $ 2 d_1 = 2 d_2 = 4 d_1 =0$, which implies  $N/4 \mid g$.  Therefore, we must have $ 2 d_1 =0$.  By considering the points $(\pm 1, 2, \pm 1)$ we arrive in the same way at $ 2 d_2 =0$.  Now we have 
      \[ g=\begin{pNiceMatrix}
    M_2 & 0 & 0 & M_2& 0\\
    0& M_2 & 0& M_2 & 0 \\
    0 & 0& d_3 & 0 & -d_3  
\end{pNiceMatrix}, \]
which is a direct sum.

Now consider the situation when each of $H_3$ and $H_4$ contain only two out of the four points. Say, $(1,1,1), \:(1,1,-1)$ are in $H_3$ and $(1,-1,1), \:(-1,1,1)$ are in $H_4$. 
It implies the generating matrix must be of the form
     \[ g=\begin{pNiceMatrix}
    d_1 & 0 & 0 & b_1 & b_4 \\
    0& d_2 & 0& \tilde{ b}_3 M_2 -b_1 & \tilde{b}_6 M_2+b_4 \\
    0 & 0& d_3 & \tilde{b}_3 M_2 & \tilde{b}_6 M_2 
\end{pNiceMatrix} \]
with $\tilde{b}_i=0$ or $1$. We cannot have $\tilde{b}_3=\tilde{b}_6$  since it would imply $d_3=0$. It is enough to consider $\tilde{b}_3=0$ and $\tilde{b}_6=1$, which gives 
$d_3=M_2$. Since the code is extended we can write 
  \[ g=\begin{pNiceMatrix}[last-row]
    d_1 & 0 & 0 & b_1 & -d_1-b_1 \\
    0& M_2+d_1+2 b_1 & 0&  -b_1 &  M_2-d_1-b_1 \\
    0 & 0& M_2 & 0 &  M_2 \\ 
H_0 & H_1 & H_2 & H_3 & H_4
\end{pNiceMatrix} \]
Consider the point $(2,1,1)$. If it is in $H_0$, then $ d_1 = M_2$. In this case, the point $(1,2,1)$ cannot belong to $H_1$ since it would give $ 4 b_1=0$ and $N/4 \mid g$, so this point must belong to $H_4$ giving $ 3 b_1 =0$.  Consider now the point $(3,2,1)$.  The only option is $(3,2,1) \in H_4$, but now this would imply  $b_1=0$ giving a few columns of zeros. 

Suppose now that $(2,1,1) \in H_4$,  i.e. $ - 3 ( d_1+b_1)=0$. Consider the point $(1,2,1)$. If it belongs to $H_1$, then $ 2 d_1 + 4 b_1 =0$, implying $b_1=d_1$ and $N/6 \mid g$. The other option is to have $(1,2,1) \in H_4$, which leads to $M_2 - 3 (d_1+b_1)=M_2=0$, a contradiction.

Now  we have considered all the possible generating matrices of interest with three rows. 
\end{proof}
\subsection{ $m=4$ case}  \label{subsec:m=4}
\setcounter{prop}{4}
\begin{prop}[Part 2] \label{prop:m=4}
        Suppose $N\geq 9$, $g$ has four rows and it generates a non-degenerate linear code $C$ that is not a direct sum. If $C$ is thin, then it can be generated by a matrix with $2$ rows of the form \eqref{familymatrix}. 
\end{prop}
\setcounter{prop}{10}
\begin{proof}
    Each pair and triple of the rows of $g$ should generate a thin linear code, therefore we have the following three options for the first three rows
    \begin{itemize}
        \item multiples of the case 5 from Table \ref{sporadictable},
        \item a direct sum with a factor $ C_{2 \Delta_2}$, 
        \item three rows with a column of zeros. 
    \end{itemize}
From the proof of Proposition \ref{m=3 1} it follows that the direct sum case can only lead to a thin code that is again a direct sum. 

Consider the situation when the first three rows is a  multiple of the generating matrix of the case 5 from Table \ref{sporadictable}. We have 
   \[ g=\begin{pNiceMatrix}
    0 & 0 & k & k & 2k \\
    2 k& 2 k & 0& 2k & 2k\\
    0 & 2k& 3k & 3k & 3k\\
    a_0 & a_1 & a_2 & a_3 & a_4
\end{pNiceMatrix} \]
over $ \zz_{ 4k}$ with $k \geq 3$. From the previous subsection we know that if $g$ has three rows and $N\geq9$, then either it is a member of the family \eqref{family1} or it is a direct sum with a factor $C_{ 2 \Delta_2}$. In both of these cases there should be three columns with $\gcdvar_i=M_2$. We see that that if consider the last three rows of $g$, it is not possible to choose $a_i$'s in such a way, that there are three columns with $\gcdvar_i=M_2$. Therefore, the linear code generated by $g$ cannot be thin. 

The only option left now is when each triple of rows of $g$ has a column of zeros, i.e. 
  \[ g=\begin{pNiceMatrix}
    d_1 & 0 & 0 & 0& -d_1\\
    0& d_2 & 0& 0 & -d_2 \\
    0 & 0& d_3 & 0 & -d_3  \\ 
0 & 0& 0& d_4 & -d_4
\end{pNiceMatrix}. \]
The fourth hyperplane must contain all the points $(\pm 1,\pm 1,\pm 1,\pm 1)$. This leads to all $d_i=M_2$, i.e. $g$ is just a multiple of the Case 1 from Table \ref{sporadictable}. 
\end{proof}

This finishes the   classification of the four-dimensional thin  simplices.

\appendix

\section{Search parameters for Table \ref{table: higher}} \label{appendix_data}
In this appendix, we explain how the data in Table \ref{table: higher} was obtained. Thin codes are constructed recursively based on the number of rows in the generating matrix, i.e. we construct a thin linear code generated by $m$ rows by starting with a thin linear code generated by $(m-1)$ rows.  For $m\geq 3$ we disregard the cases where one starts with a linear code corresponding to a lattice pyramid. This does not appear to affect the results for small values of $d$ and $N_\Delta$, and it helps to speed up the search. However, it might cause some cases to be missed. 

Checking for equivalence of two linear codes is computationally expensive. If a full equivalence check was performed, this is indicated in the last row of the table below. Otherwise, we retain only one representative from each set of linear codes that share the same $h^*$-polynomial and weight enumerator. While this omits many examples, it significantly speeds up the computation while  avoiding duplicates in the data.

In some cases, we consider matrices with entries up to $N_\Delta - 1 - \varepsilon$ instead of the full range up to $N_\Delta - 1$; this is also indicated in the table. Finally, $m$ denotes the maximum number of rows in the generating matrix that we consider.

\begin{table}[h!]
\centering
\scriptsize
\begin{tabular}{|c|cccccccccccccccccccccc|}
\hline
 
$d$ & 5 & 5 & 5 & 5 & 5 & 5 & 5 & 5 & 6 & 6 & 6 & 6 & 6 & 6 & 6 & 7 & 7 & 7 & 7 & 7 & 7 & 7 \\
$N_\Delta$ & 2 & 3 & 4 & 5 & 6 & 8 & 9 & 10 & 2 & 3 & 4 & 5 & 6 & 8 & 9 & 2 & 3 & 4 & 5 & 6 & 7 & 8 \\
m & 4 & 4 & 4 & 4 & 4 & 3 & 3 & 3 & 5 & 4 & 4 & 4 & 4 & 2 & 3 & 6 & 6 & 3 & 2 & 2 & 3 & 2 \\
$\varepsilon$ & 0 & 0 & 0 & 0 & 0 & 0 & 0 & 0 & 0 & 0 & 0 & 0 & 0 & 1 & 1 & 0 & 0 & 0 & 0 & 0 & 0 & 3 \\
= & full & full & full & full & full &  &  &  & full &  &  &  &  &  &  &  &  &  &  &  &  &  \\
\hline
\end{tabular}
\caption{Search parameters}
\label{tab:transposed}
\end{table}


\bibliographystyle{alpha}
\bibliography{zotero}
\end{document}